\documentclass{amsart}
\usepackage{amscd,amsmath,amssymb,amsthm,bm,cases,color,comment,mathrsfs}
\usepackage[pdftex]{graphicx}
\usepackage[all]{xy}

\theoremstyle{definition}
\newtheorem{definition}{Definition}[section]
\newtheorem{remark}[definition]{Remark}

\theoremstyle{plain}
\newtheorem{theorem}[definition]{Theorem}
\newtheorem{proposition}[definition]{Proposition}
\newtheorem{lemma}[definition]{Lemma}
\newtheorem{corollary}[definition]{Corollary}
\newtheorem{claim}[definition]{Claim}

\newtheorem{problem}[definition]{Problem}

\numberwithin{equation}{section}

\title[RAAGs in mapping class groups]{The RAAGs on the complement graphs of path graphs in mapping class groups}

\author[T.~Katayama]{Takuya Katayama}
\address{
(Takuya Katayama)
Department of Mathematics, Graduate School of Science, 
Hiroshima University, 
1-3-1 Kagamiyama, Higashi-Hiroshima, Hiroshima 739-8526, Japan 
}
\email{tkatayama@hiroshima-u.ac.jp}
\author[E.~Kuno]{Erika Kuno}
\address{
(Erika Kuno)
Department of Mathematics, Graduate School of Science, Osaka University,
Machikaneyama 1-1, Toyonaka, Osaka 560-0043, Japan}
\email{e-kuno@math.sci.osaka-u.ac.jp}

\date{\today}

\begin{document}

\sloppy

\begin{abstract}
In this article, we determine the function $\ell(S_{g, p})$ such that the right-angled Artin group $G(P_{m})$ is embedded in the mapping class group $\mathrm{Mod}(S_{g, p})$ if and only if $m$ is not more than $\ell(S_{g, p})$. 
Using this function and Birman--Hilden theory, we prove that $\mathrm{Mod}(S_{0, p})$ is virtually embedded in $\mathrm{Mod}(S_{g, 0})$ if and only if $p \leq 2g+2$. 
\end{abstract}

\maketitle

\section{Introduction \label{Introduction_section}}

Let $\Gamma$ be a finite graph without loops and multi-edges. 
We denote by $V(\Gamma)$ and $E(\Gamma)$ the vertex set and the edge set of $\Gamma$, respectively. 
The {\it right-angled Artin group} on $\Gamma$ is defined by the following group presentation: 
\begin{align*}
G(\Gamma)=\left\langle V(\Gamma)\mid \left[v_{i}, v_{j}\right]=1 {\rm ~if~and~only~if}~\{v_{i}, v_{j} \}\not\in E(\Gamma) \right\rangle. 
\end{align*}
For two groups $G_{1}$ and $G_{2}$, we write $G_{1}\leq G_{2}$ if there is an embedding of $G_{1}$ into $G_{2}$, that is, an injective homomorphism $G_{1} \hookrightarrow G_{2}$. 
A subgraph $\Lambda$ of a graph $\Gamma$ is called {\it induced} if any vertices $u$ and $v$ of $\Lambda$  with $\{ u, v \} \in E(\Gamma)$ span an edge in $\Lambda$. 
We write $\Gamma_{1} \leq \Gamma_{2}$ for two graphs $\Gamma_{1}$ and $\Gamma_{2}$ if $\Gamma_{1}$ is isomorphic to an induced subgraph of $\Gamma_{2}$. 
Let $S=S_{g, p}^{b}$ be the connected orientable surface of genus $g$ with $p$ punctures and $b$ boundary components, and we put $S_{g, p}=S_{g, p}^{0}$ and $S_{g}^{b}=S_{g, 0}^{b}$. 
We denote by $\mathrm{Mod}(S_{g, p}^{b})$ the {\it mapping class group} of $S_{g, p}^{b}$, the group of orientation-preserving homeomorphisms of $S_{g, p}^{b}$, fixing the punctures setwise and boundary components pointwise, up to isotopy relative to the boundary. 
The {\it curve graph} of $S$ is the graph whose vertices are isotopy classes of curves on $S$ and whose edges are given by pairs of isotopy classes of curves that can be realized disjointly. 
The {\it complement graph} $\bar{\Gamma}$ of $\Gamma$ is the graph with the vertex set $V(\Gamma)$ and the edge set 
$E(\bar{\Gamma})= \{ \{ u, v \} \mid u,v \in V(\Lambda), \ \{ u , v \} \notin E(\Lambda) \}$. 
In this paper, we use the complement graph of the curve graph, which is denoted by $\bar{\mathcal{C}}(S)$, rather than the original curve graph. 
Koberda's embedding theorem~\cite[Theorem 1.1]{Koberda12} asserts that if the Euler characteristic of an orientable surface $S$ is negative, then $\Gamma \leq \bar{\mathcal{C}}(S)$ implies $G(\Gamma) \leq \mathrm{Mod}(S)$. 
We write $B_{p}$ for the {\it braid group} on $p$ strands, which is identified with $\mathrm{Mod}(S_{0, p}^{1})$, the group of orientation preserving homeomorphisms of $S_{0, p}^{1}$, fixing the punctures setwise and the boundary pointwise, up to isotopy relative to the boundary. 
Besides, $PB_{p}$ denotes the {\it pure braid group} on $p$ strands, which is the subgroup of $B_{p}$ consisting of the elements fixing the punctures point-wise. 
In 2015 Kim-Koberda~\cite[Corollary 1.2 (1)]{Kim-Koberda15} proved that for any graph $\Gamma$ there is $p \geq 1$ such that $G(\Gamma) \leq PB_{p}$. 
Then the following problem naturally arises.

\begin{problem}
For a graph $\Gamma$ and a non-negative integer $p$, decide whether $G(\Gamma) \leq PB_{p}$ or not. 
\label{question_about_smallest_pure_braid_group}
\end{problem}

The following problem \cite[Question 1.1]{Kim-Koberda14} in the case where $g=0$ is closely related to the above problem.

\begin{problem} 
Let $\Gamma$ be a finite graph and $(g,p)$ a pair of non-negative integers. 
Decide whether $G(\Gamma) \leq \mathrm{Mod}(S_{g, p})$ or not.  
\label{question_about_simplest_surface}
\end{problem}

Let $P_{m}$ denotes the {\it path graph} on $m$ vertices, that is, the graph whose vertices is listed in the order $v_{1}, v_{2}, \ldots, v_{m}$ such that the edges are $\{v_{i}, v_{i+1}\}$ where $i=1, 2,\ldots, m-1$. 
Some of previous studies concerning embeddings of right-angled Artin groups into mapping class groups are as follows. 
Birman--Lubotzky--McCarthy~\cite[Theorem A]{Birman-Lubotzky-McCarthy83} proved that if $G$ is an abelian subgroup of $\mathrm{Mod}(S_{g, p})$, then $G$ is finitely generated and the rank of $G$ is at most $\xi(S_{g, p}):=3g-3+p$. 
We call $\xi(S_{g, p})$ the {\it topological complexity} of the surface $S_{g, p}$. 
Besides, using the chromatic numbers of graphs, Kim--Koberda \cite[Theorem 1.2]{Kim-Koberda14} proved that, for any orientable surface $S$ with negative Euler characteristic and for any positive integer $M$, there is a finite graph $\Gamma_{M}$, there is a finite graph $\Gamma_{S, M}$ with combinatorial girth $M$ such that $G(\bar{\Gamma}_{S, M}) \not\leq \mathrm{Mod}(S)$. 
Furthermore, Bering I\hspace{-.1em}V--Conant--Gaster \cite{Bering-Conant-Gaster16} introduced a graph invariant called the {\it nested complexity length} $\mathrm{NCL}(\Gamma)$, and proved that $G(\bar{\Gamma}) \leq \mathrm{Mod}(S_{g, p})$ implies $\mathrm{NCL}(\Gamma) \leq 6g-6+2p$ (see \cite[Corollary 10]{Bering-Conant-Gaster16}). 
Let $F_2$ be the free group of rank $2$, and we regard the direct product $F_2\times F_2\times \cdots\times F_2$ as $G(P_{2} \sqcup P_{2} \sqcup \cdots \sqcup P_{2})$. 
Then \cite[Lemma 5.4]{Katayama17-1} proved that $F_2 \times F_2 \times \cdots \times F_2 \leq \mathrm{Mod}(S_{g, p})$ if and only if the number of the direct factors is not more than $g - 1 + \lfloor \frac{g+p}{2} \rfloor$. 

In this paper, we prove the following theorem, which gives the answer for the path graphs to Problem~\ref{question_about_simplest_surface}.

\begin{theorem}
$G(P_{m}) \leq \mathrm{Mod}(S_{g, p})$ if and only if $m$ satisfies the following inequality. 
\begin{eqnarray*}
m \leq \left\{ \begin{array}{ll}
0 & ((g, p)\in\{(0, 0), (0, 1), (0,2), (0, 3)\}) \\
2 & ((g, p)\in\{(0, 4), (1, 0), (1,1)\}) \\
p-1 & (g=0, \ p \geq 5) \\
p+2 & (g=1, \ p \geq 2) \\
2g+p+1 & (g \geq 2). \\
\end{array} \right.
\end{eqnarray*}
\label{mcg_raag_for_path_graph}
\end{theorem}

Let $C_{m}$ be the {\it cyclic graph} on $m \geq 3$ vertices, that is, the graph that consists of $m$ vertices and the underlying space is homeomorphic to a circle. 
We also give the following partial answer to Problem  \ref{question_about_smallest_pure_braid_group}.

\begin{theorem}
Suppose that $p \geq 2$. 
Then the following hold: 
\begin{enumerate}
 \item[(1)] $G(P_{m})$ is embedded into $B_{p}$ and $PB_p$ if and only if $m$ satisfies 
\begin{eqnarray*}
m \leq \left\{ \begin{array}{ll}
p - 1 & (p=2, 3) \\
p & (p \geq 4). \\
\end{array} \right.
\end{eqnarray*}
 \item[(2)] $G(C_{m})$ is embedded into $B_{p}$ and $PB_p$ if and only if $m$ satisfies  
\begin{eqnarray*}
m \leq \left\{ \begin{array}{ll}
0 & (p = 2) \\
3 & (p = 3)  \\
p + 1 & (p \geq 4). \\
\end{array} \right.
\end{eqnarray*}
\end{enumerate}
\label{braid_group_raag}
\end{theorem}

We next describe some corollaries of Theorems \ref{mcg_raag_for_path_graph} and \ref{braid_group_raag} on embeddings of finite index subgroups of mapping class groups. 
Homomorphisms $B_{2g+1} \rightarrow \mathrm{Mod}(S_{g, 0}^{1})$ and $B_{2g+2} \rightarrow \mathrm{Mod}(S_{g, 0}^{2})$, which map the standard generators as an Artin group to the Dehn twists along a chain of interlocking simple closed curves, are injective by a theorem due to Birman--Hilden (see \cite[Chapter 9]{Farb-Margalit12}). 
In addition to this construction, the inclusions $S_{g, 0}^{1} \rightarrow S_{g+1, 0}$ and $S_{g, 0}^{2} \rightarrow S_{g+1, 0}$ induce embeddings $B_{2g+1} \leq \mathrm{Mod}(S_{g+1, 0})$ and $B_{2g+2} \leq \mathrm{Mod}(S_{g+1, 0})$ (see \cite[Theorem 4.1]{Paris-Rolfsen00} and Proof of Theorem \ref{virtual_emb_sphere_closed_surf} in Section \ref{virtual_emb_mcg_section}). 
On the other hand, Theorems \ref{mcg_raag_for_path_graph} and \ref{braid_group_raag} imply the following.

\begin{theorem}
Suppose that $g \geq 0$. 
Then the following hold. 
\begin{enumerate}
 \item[(1)] If $B_{2g+1}$ is virtually embedded into $\mathrm{Mod}(S_{g', 0})$, then $g \leq g'$. 
 \item[(2)] If $B_{2g+1}$ is virtually embedded into $\mathrm{Mod}(S_{g', 0}^{1})$, then $g \leq g'$. 
 \item[(3)] If $B_{2g+2}$ is virtually embedded into $\mathrm{Mod}(S_{g', 0})$, then $g+1 \leq g'$. 
 \item[(4)] If $B_{2g+2}$ is virtually embedded into $\mathrm{Mod}(S_{g', 0}^{2})$, then $g \leq g'$. 
\end{enumerate}
\label{for_Birman-Hilden}
\end{theorem}

In the above theorem, we say that a group $G$ is {\it virtually embedded} into a group $H$ if there is a finite index subgroup $K$ of $G$ such that $K \leq H$. 
Note that the braid groups are residually finite, and hence there are infinitely many finite index subgroups in each braid group on $\geq 2$ strands. 
We also note that theorems due to Castel \cite[Theorems 2 and 3]{Castel16} with a bit argument imply the following: 
\begin{enumerate}
 \item[$\bullet$] $B_{2g+1} \leq \mathrm{Mod}(S_{g', 0})$ implies $g+1 \leq g'$. 
 \item[$\bullet$] $B_{2g+1} \leq \mathrm{Mod}(S_{g', 0}^{1})$ implies $g \leq g'$. 
 \item[$\bullet$] $B_{2g+2} \leq \mathrm{Mod}(S_{g', 0})$ implies $g+1 \leq g'$. 
 \item[$\bullet$] $B_{2g+2} \leq \mathrm{Mod}(S_{g', 0}^{2})$ implies $g \leq g'$. 
\end{enumerate}
Therefore, a part of Theorem~\ref{for_Birman-Hilden} follows from Castel's results.
The inequalities in Theorem \ref{for_Birman-Hilden} (2), (3) and (4) are optimum by the construction of embeddings described above. 
However, Theorem \ref{for_Birman-Hilden} (1) seems to be not optimum (see Remark \ref{for_Birman-Hilden_remark}). 
To refine Theorem \ref{for_Birman-Hilden} (1), we need another argument. 
In general, we are interested in relation between surface topology and the virtual embeddability between mapping class groups. 
In this paper, using Theorem \ref{mcg_raag_for_path_graph} and Birman--Hilden theory, we obtain the following virtual embeddability result between mapping class groups.

\begin{theorem}
Let $g$ be an integer $\geq 2$. 
Then $\mathrm{Mod}(S_{0, p})$ is virtually embedded in $\mathrm{Mod}(S_{g, 0})$ if and only if $p \leq 2g+2$. 
\label{virtual_emb_sphere_closed_surf}
\end{theorem}

We note that residual finiteness of mapping class groups (\cite[Theorem 6.11]{Farb-Margalit12} and  \cite{Grossman74}) guarantees that a large supply of finite index subgroups of the mapping class groups. 

In this paper, we obtain the following corollary of Theorem \ref{mcg_raag_for_path_graph}.

\begin{corollary}
Let $g$ and $g'$ be integers $\geq 2$. 
Suppose that $\mathrm{Mod}(S_{g, p})$ is virtually embedded into $\mathrm{Mod}(S_{g', p'})$. 
Then the following inequalities (1) and (2) hold:  
\begin{enumerate}
 \item[(1)] $3g+p \leq 3g'+p'$, 
 \item[(2)] $2g+p \leq 2g'+p'$. 
\end{enumerate}
\label{rigid_mcg}
\end{corollary}

For a more precise statement of Corollary \ref{rigid_mcg}, see Theorems \ref{rigid_mcg_main} and \ref{main_thm_linear_chain}. 
The inequality (1) follows from Birman--Lubotzky--McCarthy's result, and so the new part of the above corollary is the inequality (2). 

\begin{remark}
We can see that if $(3g+p, 2g+p)=(3g'+p', 2g'+p')$, then $(g, p) = (g', p')$. 
Hence, if the following conditions are satisfied, then we have $(g, p) = (g', p')$ by Corollary \ref{rigid_mcg}: 
\begin{enumerate}
 \item[$\bullet$] the genera $g$ and $g'$ are not less than $2$, 
 \item[$\bullet$] $\mathrm{Mod}(S_{g, p})$ is virtually embedded in $\mathrm{Mod}(S_{g', p'})$, and
 \item[$\bullet$] $\mathrm{Mod}(S_{g', p'})$ is virtually embedded in $\mathrm{Mod}(S_{g, p})$. 
\end{enumerate} 

We now discuss some previous embeddability results and Corollary \ref{rigid_mcg}. 
Theorems due to Ivanov--McCarthy \cite[Theorems 3 and 4]{Ivanov-McCarthy99} assert that, embeddings between mapping class groups of connected orientable surfaces are isomorphisms induced by surface homeomorphisms, if the topological complexities of the surfaces differ by at most one and the surfaces satisfy some general conditions. 
More recently, Aramayona--Souto \cite[Corollary 1.3]{Aramayona-Souto12} proved that every non-trivial homomorphism from $\mathrm{PMod}(S_{g, p})$ to $\mathrm{PMod}(S_{g', p'})$ is induced by a surface embedding when $g \geq 6$ and $g' \leq 2g-1$ (in case $g'=2g-1$, further assume $p'=0$). 
Note that Corollary \ref{rigid_mcg} is not an immediate corollary of theorems due to Ivanov--McCarthy and Aramayona--Souto described above. 
For all $m \geq 0$, there are infinitely many pairs  of surfaces, $S_{g, p}$ and $S_{g', p'}$, such that no finite index subgroup of $\mathrm{Mod}(S_{g, p})$ is embedded into $\mathrm{Mod}(S_{g', p'})$, and the inequalities $\xi(S_{g', p'})-\xi(S_{g, p})\geq m$ and $g<6$ hold. 
In fact, by setting $S_{g, p}=S_{2, 2m+3+n}$ and $S_{g', p'} =S_{m+3, n}$, where $n$ is any non-negative integer, we obtain infinite pairs satisfying the desired properties. 
To see that no finite index subgroup of $\mathrm{Mod}(S_{2, 2m+3+n})$ is embedded into $\mathrm{Mod}(S_{m+3, n})$, use the inequality (2) in Corollary \ref{rigid_mcg}. 
\end{remark}

This paper is organized as follows. 
Section \ref{linear_chain_section} is devoted to discuss realizability of certain curve systems on surfaces in order to deduce Theorems \ref{mcg_raag_for_path_graph} and \ref{braid_group_raag}. 
In Section \ref{path_lifting_section}, we discuss embeddings of right-angled Artin groups into surface mapping class groups from combinatorial view-point for the sake of introducing an obstruction to the existence of embeddings. 
We prove Theorems \ref{mcg_raag_for_path_graph} and \ref{braid_group_raag} in Section \ref{simplest_section}. 
Virtual embeddability between mapping class groups are discussed in Section \ref{virtual_emb_mcg_section}; in particular, Theorems \ref{for_Birman-Hilden} and \ref{virtual_emb_sphere_closed_surf}, and Corollary \ref{rigid_mcg} are proved.

\subsection*{Acknowledgements}
The authors express their gratitude to Matthew Clay, Sang-hyun Kim, Thomas Koberda and Dan Margalit for helpful discussions. 
The first author thanks Makoto Sakuma and Naoki Sakata for checking proofs and suggesting a number of improvements regarding this paper.  
The second author is extremely grateful to Hisaaki Endo for his warm encouragement and helpful advice. 

\section{Linear chains on surfaces \label{linear_chain_section}}

In this section we only discuss the surfaces with boundary and without puncture in order to consider ``properly embedded arcs". 
Theorem \ref{main_thm_linear_chain}, the main result in this section, is easily translated into an equivalent result on the surface without boundary and with punctures. 
Therefore we will denote by $S_{g}^{p}$ a compact connected orientable surface of genus $g$ with $p$ boundary components and without puncture.  
An arc $\delta$ on a surface $S$ is called {\it properly embedded} if $\delta \cap \partial S = \partial \delta$. 
A properly embedded arc $\delta$ on $S$ is called {\it essential} if it is not isotopic rel $\partial \delta$ into $\partial S$.
A simple closed curve $\alpha$ on $S$ is called {\it essential} if it does not bound a disk and it is not isotopic to a boundary component of $S$. 
We denote by $N(\alpha)$ a regular neighbourhood of an arc or a curve $\alpha$, and by $\mathrm{Int}N(\alpha)$ the interior of $N(\alpha)$. 
From now on, we consider only properly embedded essential simple arcs and essential simple closed curves. 

\begin{definition} 
For two closed curves $\alpha$ and $\beta$, we denote the geometric intersection number by $i(\alpha, \beta)$. 
A sequence $\{ \alpha_1, \alpha_2, \ldots, \alpha_m \}$ of closed curves on $S_{g}^{p}$ is called a {\it linear chain} if this sequence satisfies the following. 
\begin{enumerate}
 \item[$\bullet$] Any two distinct curves $\alpha_i$ and $\alpha_j$ are non-isotopic. 
 \item[$\bullet$] Any two consecutive curves $\alpha_i$ and $\alpha_{i+1}$ are in minimal position and satisfy $i(\alpha_i, \alpha_{i+1}) > 0$. 
 \item[$\bullet$] Any two non-consecutive curves are disjoint. 
\end{enumerate}
If $\{ \alpha_1, \alpha_2, \ldots, \alpha_m \}$ is a linear chain, we call $m$ its {\it length}. 
By $\ell(S_{g}^{p})$, we denote the maximum length of the linear chains on $S_{g}^{p}$. 
Note that if $\chi(S_{g}^{p})< 0$ and $S_{g}^{p}$ is not homeomorphic to neither $S_{0}^{4}$ nor $S_{1}^{1}$, then there is a linear chain of length $m$ on $S_{g}^{p}$ if and only if $P_{m} \leq \bar{\mathcal{C}}(S_{g}^{p})$. 
Hence, the quantity $\ell(S_{g}^{p})$ is the maximum number $m$ such that $P_{m} \leq \bar{\mathcal{C}}(S_{g}^{p})$ for such surfaces $S_{g}^{p}$. 
The pair of a linear chain $\{ \alpha_1, \ldots \alpha_m \}$ and an arc $\delta$ on a surface is said to be {\it chained} if $\delta$ and the last curve $\alpha_m$ are in minimal position and not disjoint, but $\delta$ is disjoint from the other closed curves.  
\end{definition}

The main theorem in this section is the following. 

\begin{theorem}
For the maximum length $\ell(S_{g}^{p})$ of the linear chains on $S_{g}^{p}$, we have the following.
\begin{eqnarray*}
\ell(S_{g}^{p})= \left\{ \begin{array}{ll}
0 & (g=0, \ p \leq 3), \\
2 & ((g, p)\in\{(0, 4), (1, 0), (1,1)\}), \\
p-1 & (g=0, \ p \geq 5), \\
p+2 & (g=1, \ p \geq 2), \\
2g+p+1 & (g \geq 2). \\
\end{array} \right. 
\end{eqnarray*}
\label{main_thm_linear_chain}
\end{theorem}

By induction on the ordered pair $(g, p)$ and a surface cutting argument, we prove this theorem. 
To proceed the induction, we need the following two lemmas. 

\begin{lemma}\label{cutting_surfaces}
The following (1) and (2) hold.  
\begin{enumerate}
 \item[(1)] If $\alpha$ is a closed curve on $S_{g}^{p}$, then 
 \begin{eqnarray*}
S_{g}^{p} \setminus \mathrm{Int}N(\alpha) \cong \left\{ \begin{array}{ll}
S_{g-1}^{p+2} \ & (\alpha: \mbox{ non-separating}), \\
S_{g_1}^{p_1} \sqcup S_{g_2}^{p_2}  \ & (\alpha: \mbox{ separating}), \\ 
\end{array} \right.
\end{eqnarray*}
where $g_{1}$, $p_{1}$, $g_{2}$ and $p_{2}$ are natural numbers satisfying 
 \begin{itemize}
 \item[$\bullet$] $g_1 + g_2 = g$, $g_1 \geq 0$, $g_2 \geq 0$, 
 \item[$\bullet$] $p_1 + p_2 = p + 2$, $p_1 \geq 1$, $p_2 \geq 1$,
 \item[$\bullet$] if $g_i=0$, then $p_i \geq 3$ {\rm (}$i=1,2${\rm)}.
 \end{itemize}
 
 \item[(2)] If $\delta$ is an arc on $S_{g}^{p}$, then 
 \begin{eqnarray*}
S_{g}^{p} \setminus \mathrm{Int}N(\delta) \cong \left\{ \begin{array}{ll}
S_{g-1}^{p+1} \ & (\delta: \mbox{ non-separating}), \\
S_{g_1}^{p_1} \sqcup S_{g_2}^{p_2}  \ & (\delta: \mbox{ separating}), \\ 
\end{array} \right.
\end{eqnarray*}
where $g_{1}$, $p_{1}$, $g_{2}$ and $p_{2}$ are natural numbers satisfying 
 \begin{itemize}
 \item[$\bullet$] $g_1 + g_2 = g$, $g_1 \geq 0$, $g_2 \geq 0$,
 \item[$\bullet$] $p_1 + p_2 = p + 1$, $p_1 \geq 1$, $p_2 \geq 1$,
  \item[$\bullet$] if $g_i=0$, then $p_i \geq 2$ {\rm (}$i=1,2${\rm )}.
 \end{itemize}
\end{enumerate}
\end{lemma}

\begin{lemma}
Let $S=S_{g}^{p}$ be a compact connected orientable surface. 
Then the following (1) and (2) hold.  
\begin{enumerate}
 \item[(1)] Let $\{ \alpha_1, \ldots, \alpha_m \}$ be a linear chain on $S$. 
 Then $\{ \alpha_1, \ldots, \alpha_{m-2} \}$ is a linear chain on the connected component $S'$ of $S \setminus \mathrm{Int}N(\alpha_m)$.
Moreover, $(\{ \alpha_1, \ldots, \alpha_{m-2} \}, \delta)$ is a chained pair on $S'$, where $\delta$ is any connected component of $\alpha_{m-1} \cap S'$ intersecting with $\alpha_{m-2}$. 
 \item[(2)] Let $(\{ \alpha_1, \ldots, \alpha_m \}, \delta)$ be a chained pair on $S$. 
 \begin{itemize}
  \item[(i)] If some $\alpha_i$ is isotopic into $\partial S \cup \delta$, then $i=1$ and $m=2$. 
  \item[(ii)] If $m \geq 3$, then $(\{ \alpha_1, \ldots, \alpha_{m-1}\}, \delta')$ is a chained pair on a connected component $S'$ of $S\setminus \mathrm{Int}N(\delta)$, where $\delta'$ is any connected componen of $\alpha_{m}\cap S'$ intersecting with $\alpha_{m-1}$.
 \end{itemize} 
\end{enumerate}
\label{key_property_of_chains}
\end{lemma}

We can prove Lemma~\ref{cutting_surfaces} by computing Euler characteristic, and so we only prove Lemma~\ref{key_property_of_chains}. 

\begin{proof}[Proof of Lemma~\ref{key_property_of_chains}]
(1) The first assertion follows from the fact that none of essential closed curves $\alpha_1, \ldots, \alpha_{m-2}$ is isotopic to $\alpha_{m}$, and hence $\alpha_1, \ldots, \alpha_{m-2}$ are essential in $S \setminus \mathrm{Int}N(\alpha_m)$. 
The second assertion follows from the fact that $\alpha_{m-1}$ and $\alpha_{m}$ are in minimal position  and hence any component $\delta$ of $\alpha_{m-1} \cap S'$ is essential. 

(2-i) Suppose that a curve $\alpha_i$ is isotopic into $\partial S \cup \delta$. 
Then according to whether $\partial \delta$ lies in a single boundary component $C$ or $\partial \delta$ joins two boundary components $C_1$ with $C_2$, we have the following. 
\begin{enumerate}
 \item[(a)] If $\delta$ joins a boundary component $C$
to itself, then there is an annulus $A$ in $S$ such that $ \partial A \subset \alpha_i \sqcup (\delta \cup  \partial S)$ (see Figure \ref{annulus_disk} (a)). 
 \item[(b)] If $\delta$ joins two boundary components $C_1$ and $C_2$, then there is a twice-holed disk $D$ such that $ \partial D = \alpha_i \sqcup C_1 \sqcup C_2$ (see Figure \ref{annulus_disk} (b)). 
\end{enumerate}
\begin{figure}
\centering
\includegraphics[clip, scale=0.35]{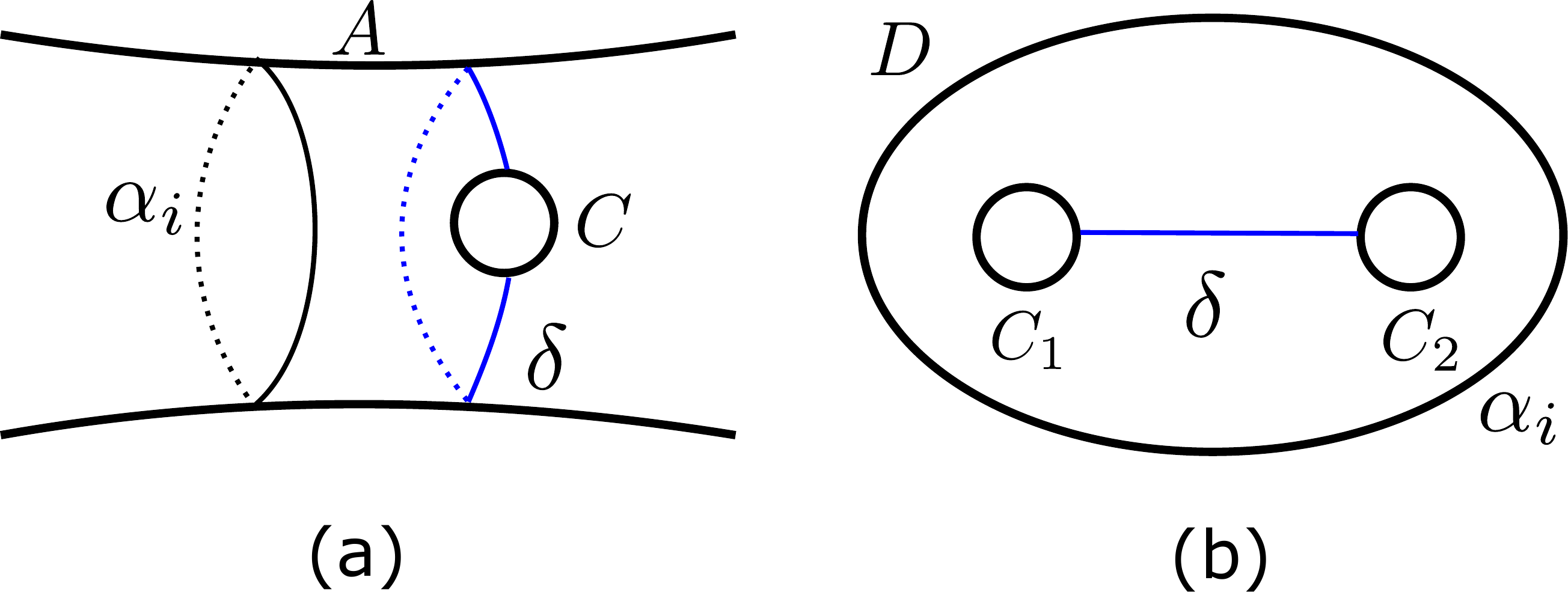}
\caption{\label{annulus_disk}}
\end{figure}
In each case any closed curve in $S$ which intersects with $\alpha_i$ non-trivially and minimally must intersect with $\delta$. 
Hence, we have $i=1$ and $m=2$. 

(2-ii) Suppose $m \geq 3$. 
Then by (2-i), $\alpha_1, \ldots, \alpha_{m-1}$ are essential in $S'$, and so $\{ \alpha_1, \ldots, \alpha_{m-1}\}$ is a linear chain on $S'$. 
Moreover, since $\alpha_{m-1}$ and $\alpha_{m}$ are in minimal position, $\delta'$ is essential in $S'$. 
Thus, $(\{ \alpha_1, \ldots, \alpha_m-1\}, \delta')$ is a chained pair on $S'$. 
\end{proof}


Since $S_{0}^{3}$ does not contain an essential simple closed curve and since any two mutually non-isotopic curve must intersect in $S_{0}^{4}$ (see Figure \ref{lc_0n}), we obtain the following lemma.

\begin{lemma}
We have $\ell(S_{0}^{3}) = 0$ and $\ell(S_{0}^{4}) = 2$. 
\label{0-4}
\end{lemma}

\begin{figure}
\centering
\includegraphics[clip, scale=0.35]{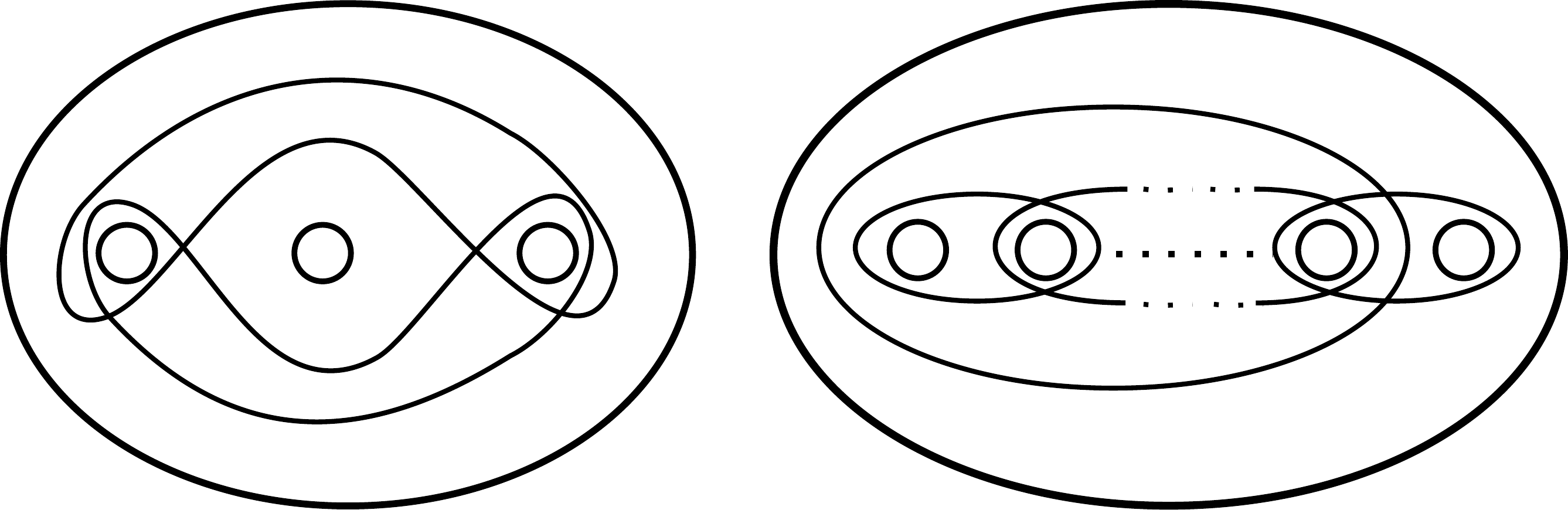}
\caption{Linear chains on $S_{0}^{4}$ (left-hand side) and $S_{0}^{p}$ (right-hand side), where the small circles in the both left and right-hand sides are three and $p-1$ boundary components of $S_{0}^{4}$ and $S_{0}^{p}$, respectively. \label{lc_0n}}
\end{figure}

For $p\geq 5$, we will prove the following lemma.

\begin{lemma}\label{key_property_of_chains_in_punctured_disks}
Suppose $p \geq 5$. 
Then the following (1) and (2) hold.  
\begin{enumerate}
 \item[(1)] $\ell(S_{0}^{p}) = p-1$. 
 \item[(2)] If the pair of a linear chain $\{ \alpha_1, \ldots, \alpha_{m} \}$ and an arc $\delta$ is chained on $S_{0}^{p}$, then $m \leq p -2$. 
\end{enumerate}
\label{0-n}
\end{lemma}

\begin{proof}[Proof of Lemma~\ref{key_property_of_chains_in_punctured_disks}]
The picture on the right hand side of Figure \ref{lc_0n} shows $ \ell(S_{0}^{p}) \geq p-1$. 

Case $p=5$. 

(1) Suppose that $\{ \alpha_1, \ldots, \alpha_m \}$ is a linear chain on $S_{0}^{5}$. 
Then we obtain a connected component $S_{g'}^{p'}$ of  $S_{0}^{5} \setminus \mathrm{Int}N(\alpha_m)$, which contains the linear chain $\{ \alpha_1, \ldots, \alpha_{m-2} \}$ by Lemma \ref{key_property_of_chains} (1). 
By Lemma \ref{cutting_surfaces} (1), $g'=0$ and $p' \leq 4$.  
Hence, Lemma \ref{0-4} implies $m-2 \leq 2$, namely, $m \leq 4$. 
Consequently, we have $\ell(S_{0}^{5}) = 4$. 

(2) Suppose that the pair of a linear chain $\{ \alpha_1, \ldots, \alpha_m \}$ and an arc $\delta$ is chained on $S_{0}^{5}$. 
We assume $m \geq 3$ and prove $m = 3$. 
By Lemma \ref{key_property_of_chains} (2), a connected component $S_{g'}^{p'}$ of $S_{0}^{5} \setminus \mathrm{Int}N(\delta)$, contains a linear chain of length $m-1$. 
By Lemma \ref{cutting_surfaces} (2), $g' = 0$ and $p' \leq 4$. 
Therefore Lemma \ref{0-4} implies that $m-1 \leq 2$, namely, $m \leq 3$, as required. 

Case $p \geq 6$. 

(1) Suppose that $\{ \alpha_1, \ldots, \alpha_m \}$ is a linear chain on $S_{0}^{p}$. 
Then a connected component $S_{g'}^{p'}$ of $S_{0}^{p} \setminus \mathrm{Int}N(\alpha_m)$ contains the chained pair $( \{ \alpha_1, \ldots, \alpha_{m-2} \} , \delta )$ by Lemma \ref{key_property_of_chains} (1). 
Since $g' = 0$ and $p' \leq p - 1$ by Lemma \ref{cutting_surfaces} (1), the induction hypothesis implies an inequality $m-2 \leq (p - 1) - 2 $ which means $m \leq p-1$. 

(2) We may assume $m \geq 3$. 
A connected component $S_{g'}^{p'}$ of $S_{0}^{p} \setminus \mathrm{Int}N(\delta)$ contains a chained pair $(\{ \alpha_1, \ldots, \alpha_{m-1} \} , \delta')$ by Lemma \ref{key_property_of_chains} (2), where $\delta'$ is an arc derived from $\alpha_m$. 
Since $g'=0$ and $p' \leq p-1$, by the induction hypothesis we have an inequality $m-1 \leq (p-1) - 2$ which means $m \leq p - 2$. 
\end{proof}

In order to compute $\ell(S_{1}^{p})$, we first deal with the following two exceptions. 

\begin{lemma}
We have $\ell(S_{1}^{0})= \ell(S_{1}^{1})=2$. 
\label{1-1}
\end{lemma}
\begin{proof}
We prove the assertion only in the case where $p=1$, because the case where $p=0$ can be treated similarly.  
In Figure \ref{lc_1n}, we can see the inequality $\ell(S_{1}^{1}) \geq 2$. 
\begin{figure}
\centering
\includegraphics[clip, scale=0.35]{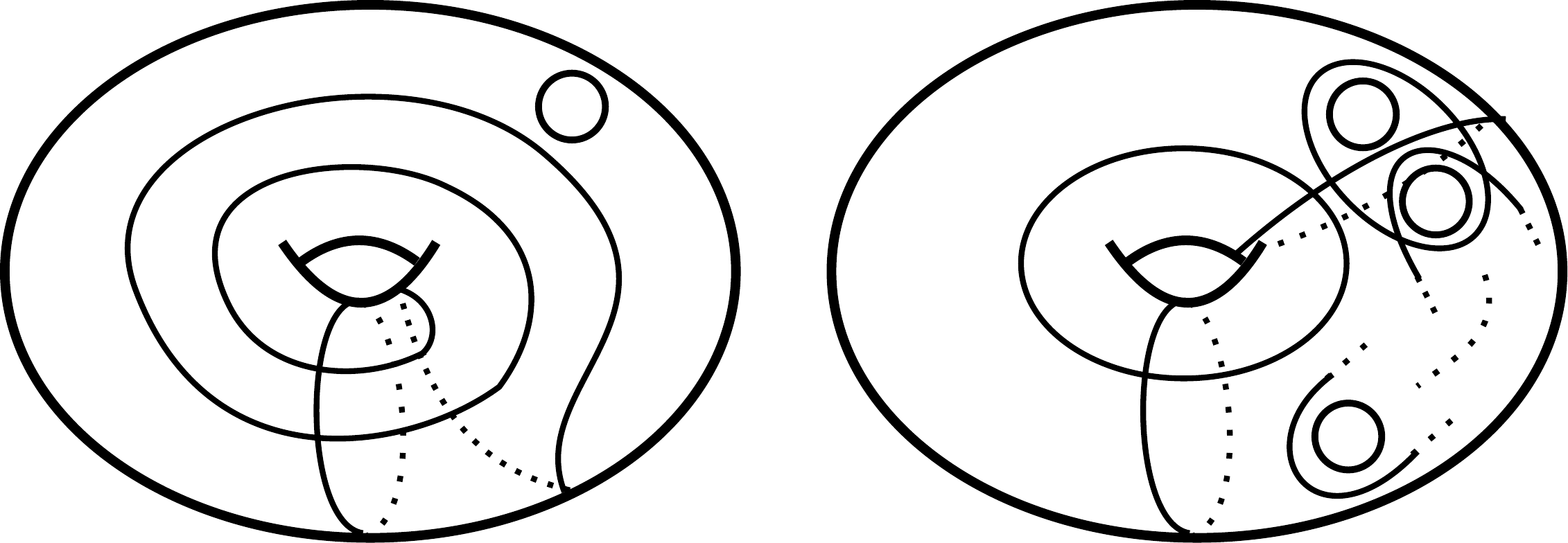}
\caption{Linear chains on $S_{1}^{1}$ (left-hand side) and $S_{1}^{p}$ (right-hand side), where the small circles in the both of $S_{1}^{1}$ and $S_{1}^{p}$ are the boundary components of them.}
\label{lc_1n}
\end{figure}
To see $\ell(S_{1}^{1}) \leq 2$, suppose that $S_{1}^{1}$ contains a linear chain of length $m$, $\{ \alpha_1, \ldots, \alpha_m \}$. 
Note that $\alpha_m$ must be a non-separating closed curve. 
Then the surface $S = S_{1}^{1} \setminus \mathrm{Int}N(\alpha_m)$ contains an essential simple closed curve $\alpha_{m-2}$ by Lemma \ref{key_property_of_chains} (1). 
However, since $S$ is homeomorphic to $S_{0}^{3}$, Lemma \ref{0-4} implies that $\alpha_{m-2}$ does not exist. 
Thus we have $m \leq 2$, and so $\ell(S_{1}^{1}) = 2$.
\end{proof}

\begin{lemma}
Suppose $p \geq 2$. 
Then the following (1) and (2) hold.  
\begin{enumerate}
 \item[(1)] We have $\ell(S_{1}^{p}) = p+2$. 
 \item[(2)] If the pair of a linear chain $\{ \alpha_1, \ldots, \alpha_{p} \}$ and an arc $\delta$ is chained on $S_{1}^{p}$, then $m \leq p+1$. 
\end{enumerate}
\label{1-n}
\end{lemma}
\begin{proof}
Figure \ref{lc_1n} shows $\ell(S_{1}^{p}) \geq p + 2$ for any $p \geq 2$. 
We shall prove the remainder of the assertion of Lemma \ref{1-n} by induction on $p$. 

Case $p=2$.

(1) Notice that any linear chain of length $5$ on $S_{1}^{2}$ induces a linear chain of length $3$ on a surface isomorphic to one of $S_{0}^{3}$, $S_{0}^{4}$ and $S_{1}^{1}$ by Lemmas \ref{cutting_surfaces} (1) and  \ref{key_property_of_chains} (1). 
However, by Lemmas \ref{0-4} and \ref{1-1}, these surfaces does not contain a linear chain of length $3$. 
Hence, we have $\ell(S_{1}^{2}) = 4$. 

(2) Suppose that $(\{ \alpha_1, \ldots, \alpha_{m} \} , \delta)$ ($m \geq 3$) is the chained pair of a linear chain $\{ \alpha_1, \ldots, \alpha_m \}$ and an arc $\delta$ on $S_{1}^{2}$. 
Then a connected component $S$ of $S_{1}^{2} \setminus \mathrm{Int}N(\delta)$ contains the chained pair of a linear chain of length $m-1$ and an arc derived from $\alpha_{m}$ by Lemma \ref{key_property_of_chains} (2). 
The surface $S$ is homeomorphic to either $S_{0}^{3}$ or $S_{1}^{1}$ by Lemma \ref{cutting_surfaces} (2). 
By Lemmas \ref{0-4} and \ref{1-1}, in either case, we have $m-1 \leq 2$. 
Hence, $m \leq 3$. 

Case $p \geq 3$. 

(1) Suppose that $\{ \alpha_1, \ldots, \alpha_m \}$ is a linear chain on $S_{1}^{p}$. 
Then a connected component $S_{g'}^{p'}$ of $S_{1}^{p} \setminus \mathrm{Int}N(\alpha_{m})$ contains the chained pair of a linear chain of length $m-2$ and an arc derived from $\alpha_{m-1}$ by Lemma \ref{key_property_of_chains} (1). 
By Lemma \ref{cutting_surfaces} (1), we have either $g'=0, \ p' \leq p + 2$ or $g'=1, \ p' \leq p - 1$. 
In any case, Lemmas \ref{0-4} and \ref{0-n} (2) and the induction hypothesis imply that $m-2 \leq p$, namely, $m \leq p+2$. 

(2) Suppose that $(\{ \alpha_1, \ldots, \alpha_m \}, \delta)$ ($m \geq 3$) is a chained pair on $S_{1}^{p}$. 
Let $S_{g'}^{p'}$ be a connected component of $S_{1}^{p} \setminus \mathrm{Int}N(\delta)$, which contains a chained pair $(\{ \alpha_1, \ldots, \alpha_{m-1} \}, \delta')$, where $\delta'$ is an arc derived from $\alpha_m$ by Lemma \ref{key_property_of_chains} (2). 
Then, by Lemma \ref{cutting_surfaces} (2), we have either $g'=0, \ p' \leq p + 1$ or $g'=1, \ p' \leq p - 1$. 
Hence, Lemmas \ref{0-4}, \ref{0-n} (2) and \ref{1-1} and the induction hypothesis imply that $m-1 \leq p$. 
Thus $m \leq p + 1$, as desired. 
\end{proof}

\begin{lemma}
Suppose that $g \geq 2$ and $p \geq 0$.  
Then following (1) and (2) hold: 
\begin{enumerate}
 \item[(1)] $\ell(S_{g}^{p}) = 2g+p+1$. 
 \item[(2)] If the pair of a linear chain $\{ \alpha_1, \ldots, \alpha_{m} \}$ and an arc $\delta$ is chained on $S_{g}^{p}$, then $m \leq 2g+p-1$. 
\end{enumerate}
\label{2-n}
\end{lemma}
\begin{proof}
We first prove the assertion in the case where $g = 2 $. 

Case $g=2, p=0$. 

(1) Figure \ref{lc_20_21} shows $\ell(S_{2}^{0}) \geq 5$. 
\begin{figure}
\centering
\includegraphics[clip, scale=0.35]{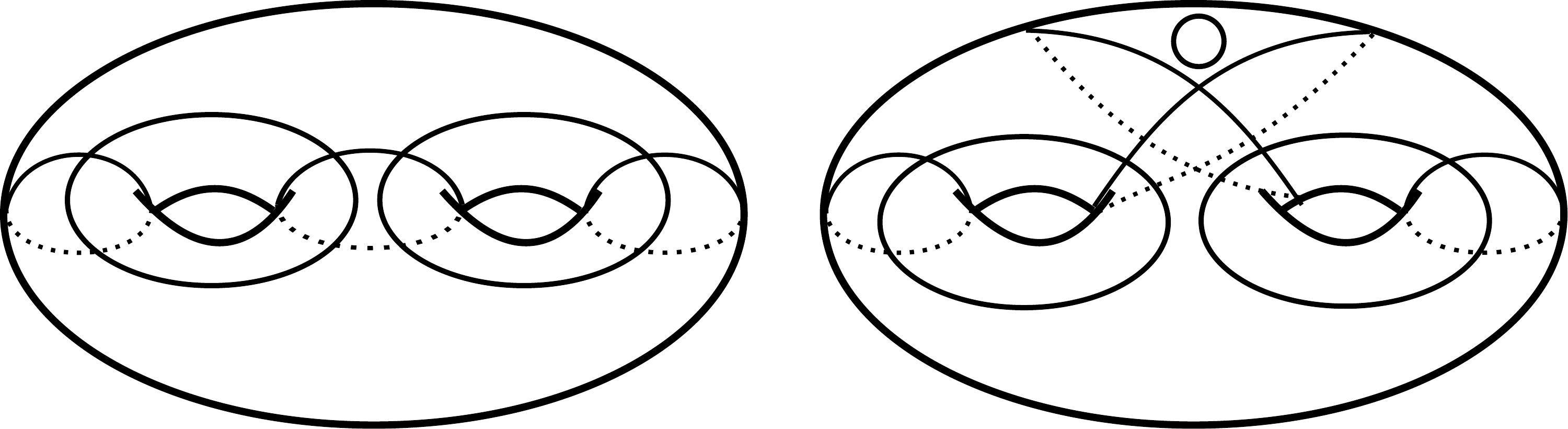}
\caption{Linear chains on $S_{2}^{0}$ (left-hand side) and $S_{2}^{1}$ (right-hand side), where the small circle in the right-hand side is the boundary component of $S_{2}^{1}$. \label{lc_20_21}}
\end{figure}
To see the converse inequality, suppose that $\{ \alpha_1, \ldots, \alpha_m \}$ is a linear chain on $S_{2}^{0}$. 
Then a connected component $S$ of $S_{2}^{0} \setminus \mathrm{Int}N(\alpha_m)$ contains a chained pair $(\{ \alpha_1, \ldots, \alpha_{m-2} \}, \delta)$ by Lemma \ref{key_property_of_chains} (1). 
Since $S$ is homeomorphic to either $S_{1}^{1}$ or $S_{1}^{2}$, we have $m-2 \leq 3$ by Lemma \ref{1-n} (2). 
Hence, $m \leq 5$. 

Since $S_{2}^{0}$ is a closed surface, there is no properly embedded arc on this surface, and therefore the assertion (2) makes no sense in this case. 

Case $g=2, p=1$.
 
(1) Figure \ref{lc_20_21} shows $\ell(S_{2}^{1}) \geq 6$. 
To prove $\ell(S_{2}^{1}) \leq 6$, suppose that $\{ \alpha_1, \ldots, \alpha_m \}$ is a linear chain on $S_{2}^{1}$. 
Then a connected component $S$ of $S_{2}^{1} \setminus \mathrm{Int}N(\alpha_m)$ contains the chained pair of a linear chain $(\{ \alpha_1, \ldots, \alpha_{m-2} \}, \delta)$ by Lemma \ref{key_property_of_chains} (1). 
Since $S$ is homeomorphic to the surface of genus $1$ with $p' (\leq 3)$ boundary components, Lemmas \ref{1-1} and \ref{1-n} (2) imply $m-2 \leq 4$, and therefore $m \leq 6$. 

(2) We may assume that $m \geq 3$. 
Let $S_{g'}^{p'}$ be a connected component of $S_{2}^{1} \setminus \mathrm{Int}N(\delta)$, which contains the chained pair $\{ \alpha_1, \ldots, \alpha_{m-1} \}, \delta')$ by Lemma \ref{key_property_of_chains} (2). 
Then, by Lemma \ref{cutting_surfaces} (2), we have $g'= 1, \ p' \leq 2$. 
Hence, Lemmas \ref{1-1} and \ref{1-n} (2) imply $m-1 \leq 3$, and therefore $m \leq 4$. 

Case $g = 2, \ p \geq 2$. 

(1) Figure \ref{lc_2n} shows $\ell(S_{2}^{p}) \geq 5 + p  = 2g + p + 1$. 
\begin{figure}
\centering
\includegraphics[clip, scale=0.40]{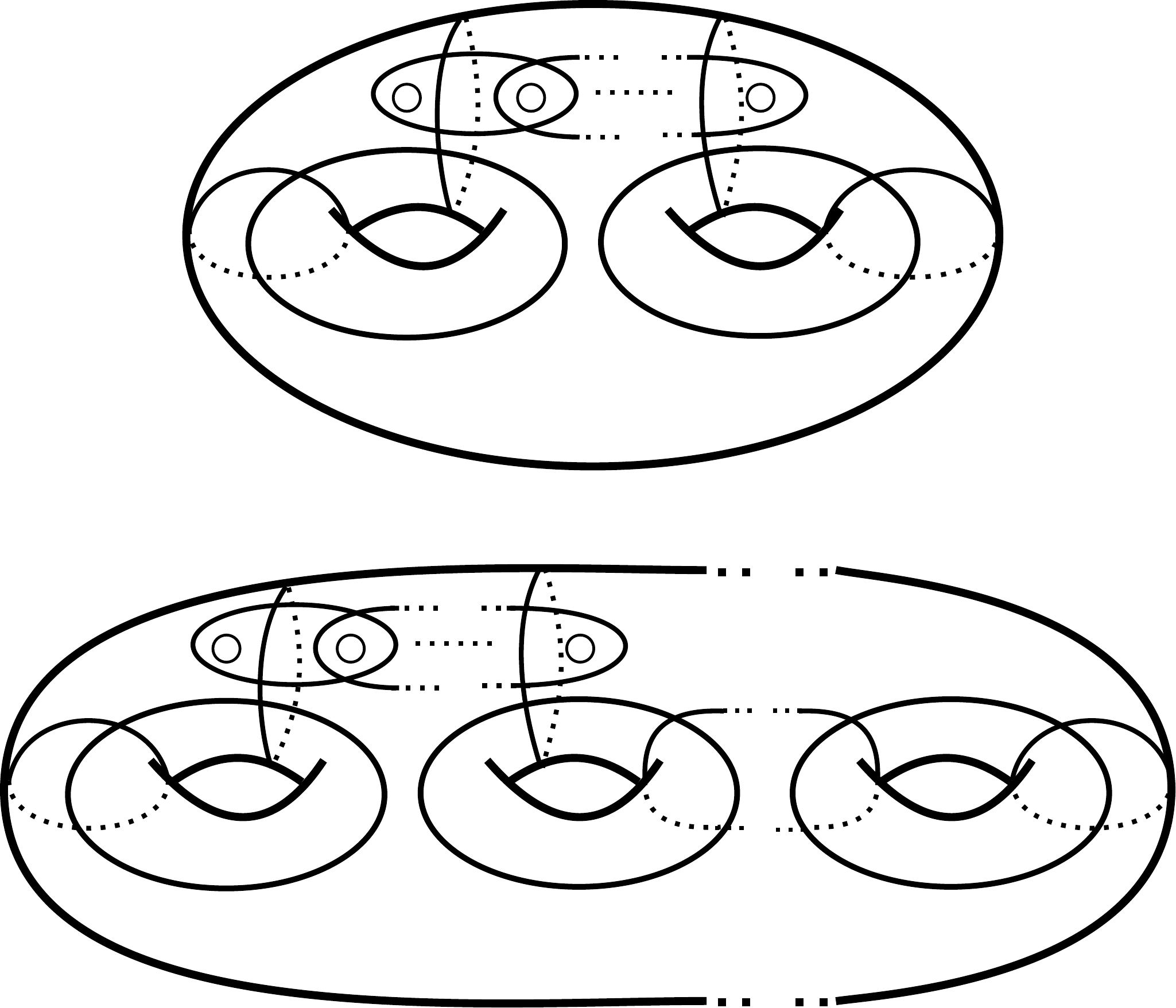}
\caption{Linear chains on $S_{2}^{p}$ and $S_{g}^{p}$ ($g \geq 3$), where the small circles are the boundary components. \label{lc_2n}}
\end{figure}
We now suppose that $\{ \alpha_1, \ldots, \alpha_m \}$ is a linear chain on $S_{2}^{p}$. 
Let $S_{g'}^{p'}$ be a connected component of $S_{2}^{p} \setminus \mathrm{Int}N(\alpha_m)$, which contains the chained pair $(\{ \alpha_1, \ldots, \alpha_{m-2} \}, \delta)$ in the assertion of Lemma \ref{key_property_of_chains} (1). 
By Lemma \ref{cutting_surfaces} (1), we have one of the following: 
 \begin{itemize}
 \item[$\bullet$] $g'=0, \ p' \leq p + 1$, 
 \item[$\bullet$] $g'=1, \ p' \leq p + 2$, or
 \item[$\bullet$] $g'=2, \ p' \leq p - 1$. 
 \end{itemize}
Now, Lemmas \ref{0-4}, \ref{0-n} (2), \ref{1-1}, \ref{1-n} (2) and the induction hypothesis imply that $m-2 \leq p+3$ in any case, namely, $m \leq p+5$. 

(2) Let $S_{g'}^{p'}$ be a connected component of $S_{2}^{p} \setminus \mathrm{Int}N(\delta)$ contains the chained pair $(\{ \alpha_1, \ldots, \alpha_{m-1} \}, \delta')$ in the assertion of  Lemma \ref{key_property_of_chains} (2). 
Then, by Lemma \ref{cutting_surfaces} (2), one of the following holds: 
 \begin{itemize}
 \item[$\bullet$] $g' \leq 0, \ p' \leq p$, 
 \item[$\bullet$] $g'=1, \ p' \leq p + 1$, or
 \item[$\bullet$] $g'=2, \ p' \leq p - 1$. 
 \end{itemize} 
Lemmas \ref{0-4}, \ref{0-n} (2), \ref{1-1}, \ref{1-n} (2) and the induction hypothesis show that $m-1 \leq p+2$, which means $m \leq p+3$. 

Case $g \geq 3$. 

(1) Figure \ref{lc_2n} shows $\ell(S_{g}^{p}) \geq 2g + p + 1$. 
To see the converse inequality, suppose that $\{ \alpha_1, \ldots, \alpha_m \}$ is a linear chain on $S_{g}^{p}$. 
Let $S_{g'}^{p'}$ be the connected component of the surface obtained from $S_{g}^{p}$ by cutting along the essential simple closed curve $\alpha_m$, which contains a chained pair $(\{ \alpha_1, \ldots, \alpha_{m-2} \}, \delta)$, in the assertion of Lemma \ref{key_property_of_chains} (1). 
Then, by Lemma \ref{cutting_surfaces} (1), one of the following holds: 
 \begin{itemize}
 \item[$\bullet$] $g' \leq g-2, \ p' \leq p + 1$, 
 \item[$\bullet$] $g'=g-1, \ p' \leq p + 2$, or
 \item[$\bullet$] $g'=g, \ p' \leq p - 1$. 
 \end{itemize}
For the pairs $(g', p')$ of the genera $g'$ and the numbers $p'$ of the punctures, we consider the function
\begin{eqnarray*}
f(g', p')= \left\{ \begin{array}{ll}
p'-2 &  (g'=0), \\
p'+1 & (g'=1), \\
2g'+p'-1 & (g'\geq 2). \\
\end{array} \right. 
\end{eqnarray*}
Note that this function represents the maximum length of the linear chains in the chained pairs on the surfaces $S_{g', p'}$. 
If $(g', p')$ runs all possible pairs of non-negative integers, the maxima of $f$ for $g'=0, 1, 2, \ldots, g$ ($g\geq 3$) are
\begin{itemize}
\item[$\bullet$] $f(0, p+1)=p-1$,
\item[$\bullet$] $f(1, p+1)=p+2$, 
\item[$\bullet$] $f(g', p') \leq 2g+p-4$ (when $g' \leq g-2$, $p' \leq p+1$), 
\item[$\bullet$] $f(g-1, p+2) = 2(g-1)+(p+ 2)-1=2g+p-1$ and
\item[$\bullet$] $f(g, p-1) = 2g + p - 2$, 
\end{itemize}
respectively.
Since $g \geq 3$, the value $f(g-1, p+2)=2g+p -1 \ (\geq p + 5)$ is the largest number among all such values. 
Thus, we have $m-2 \leq 2g + p - 1$, namely, $m \leq 2g + p + 1$, as desired. 

(2) Let $S_{g'}^{p'}$ be a connected component of  $S_{g}^{p} \setminus \mathrm{Int}N(\delta)$, which contains a chained pair $(\{ \alpha_1, \ldots, \alpha_{m-1} \}, \delta')$, in the assertion of Lemma \ref{key_property_of_chains} (1). 
Then, by Lemma \ref{cutting_surfaces} (2), one of the following holds: 
 \begin{itemize}
 \item[$\bullet$] $g' \leq g-2, \ p' \leq p$, 
 \item[$\bullet$] $g'=g-1, \ p' \leq p + 1$ and 
 \item[$\bullet$] $g'=g, \ p' \leq p-1 $. 
 \end{itemize}
If $(g', p')$ runs all possible pairs of numbers, the maxima of the function $f$ defined above for $g' =0, 1, 2, \ldots, g$ are
\begin{itemize}
\item[$\bullet$] $f(0, p) = p - 2$,
\item[$\bullet$] $f(1, p) = p + 1$, 
\item[$\bullet$] $f(g', p') \leq 2g+p-5$ (when $g' \leq g-2$, $p' \leq p$), 
\item[$\bullet$] $f(g-1, p+1)=2(g-1)+p+1-1=2g+p-2$ and
\item[$\bullet$] $f(g, p-1) = 2g + p - 2$, 
\end{itemize}
respectively. 
Since $g \geq 3$, the value $f(g-1, p+2) = 2g + p -2 \ (\geq p+4)$ is the largest number among all such values. 
Thus we have $m-1 \leq 2g + p - 2$, namely, $m \leq 2g + p - 1$, as desired. 
\end{proof}

By combining Lemmas \ref{0-4}, \ref{0-n}, \ref{1-1}, \ref{1-n} and \ref{2-n}, we obtain Theorem \ref{main_thm_linear_chain}.

\section{Path lifting \label{path_lifting_section}}

From Section~\ref{path_lifting_section}, we will mainly consider the orientable surfaces of genus $g$ with $p$ punctures and denote it by $S_{g, p}$. 
In this section we discuss embeddings of right-angled Artin groups from combinatorial view-point. 
Main purpose of this section is to introduce path lifting obstruction (Lemma \ref{path-lifting_lemma}) on embeddings of right-angled Artin groups into surface mapping class groups. 

We first note a basic fact on right-angled Artin groups without a proof.

\begin{lemma}
Let $\Gamma$ be a finite graph. 
Then the following hold. 
\begin{enumerate}
 \item[(1)] For any non-zero integer $m$, the power endomorphism $f \colon G(\Gamma) \rightarrow  G(\Gamma);$ $v \mapsto v^m$ is injective. 
 \item[(2)] If a group $K$ contains the right-angled Artin group $G(\Gamma)$, then any finite index subgroup of $K$ contains $G(\Gamma)$. 
\end{enumerate}
\label{power-homomorphism_injective}
\end{lemma}

We now restate an important result due to Kim--Koberda \cite[Lemma 2.3]{Kim-Koberda14} on embeddings of right-angled Artin groups into surface mapping class groups. 

\begin{theorem}
Suppose that $G(\Lambda) \hookrightarrow \mathrm{Mod}(S_{g, p})$ with $\chi(S_{g, p}) < 0$. 
Then there is an embedding $\psi \colon G(\Lambda) \hookrightarrow \mathrm{Mod}(S_{g, p})$ satisfying the following property: 
\begin{enumerate}
 \item[$\bullet$] for each vertex $v \in V(\Lambda)$, there are Dehn twists $T_{v, 1}, \ldots, T_{v, m_v}$ and non-zero numbers $e(v, 1), \ldots, e(v, m_v)$ such that 
$$\psi(v) = [T_{v, 1}]^{e(v, 1)} \cdots [T_{v, m_v}]^{e(v, m_v)},$$ 
where $[T_{v, i}]$ and $[T_{v, j}]$ are commutative for all $1 \leq i \leq j \leq m_v$. 
\end{enumerate}
\label{Kim-Koberda_obst}
\end{theorem}
\begin{proof}
By \cite[Proof of Lemma 2.3]{Kim-Koberda14}, we have an embedding $\psi_0 \colon G(\Lambda) \hookrightarrow \mathrm{Mod}(S_{g, p})$ satisfying the following properties. 
\begin{enumerate}
 \item[$\bullet$] For any vertex $v \in V(\Lambda)$, there are self-homeomorphisms $f_{v, 1}, \ldots, f_{v, m_v}$ of $S_{g, p}$ and and non-zero numbers $e'(v, 1), \ldots, e'(v, m_v)$ such that 
 $$\psi_0 (v) = [f_{v, 1}]^{e'(v, 1)} \cdots [f_{v, m_v}]^{e'(v, m_v)}, $$  
 and that mapping classes $[f_{v, i}]$ and $[f_{v, j}]$ are commutative for all $1 \leq i \leq j \leq m_v$. 
 \item[$\bullet$] The set of mapping classes $X:= \{ [f_{v, i}] \mid v \in  V(\Lambda), \ 1 \leq i \leq m_v \}$ induces a set of mutually non-isotopic closed curves $Y:=\{ \alpha_{v, i} \mid v \in  V(\Lambda), \ 1 \leq i \leq m_v  \}$ in the sense that the map  defined by $[f_{v, i}] \mapsto \alpha_{v, i}$ from $X$ to $Y$ induces an embedding of the anti-commutation graph of $X$ into $\bar{\mathcal{C}}(S_{g, p})$ as an induced subgraph. 
 Here, the anti-commutation graph $\mathcal{X}$ of given elements $g_1, \ldots, g_n$ in a group $G$ is the graph such that $V(\mathcal{X})= \{ g_1, \ldots, g_n\}$, and two vertices $g_i$ and $g_j$ span an edge if and only if $g_i$ and $g_j$ are not commutative in $G$. 
 By $\mathcal{Y}$ we denote the subgraph of $\bar{\mathcal{C}}(S_{g, p})$ induced by $Y$. 
 \item[$\bullet$] The natural homomorphism 
 $$\eta \colon G(\mathcal{Y}) \rightarrow \mathrm{Mod}(S_{g, p}); \ \alpha_{v, i} \mapsto [f_{v, i}]$$ 
 is an embedding. 
\end{enumerate}

Consider the set of the Dehn twists $\{ [T_{v, i}] \mid 1 \leq i \leq m_v, \  v \in  V(\Lambda) \}$, where $T_{v, i}$ is the Dehn twist along $\alpha_{v, i}$. 
Note that the anti-commutation graph $\mathcal{X}_{T}$ of $\{ [T_{v, i}] \mid 1 \leq i \leq m_v, \  v \in  V(\Lambda) \}$ coincides with $\mathcal{Y}$ by identifying $[T_{v, i}]$ with $\alpha_{v, i}$. 
Now Koberda's embedding Theorem asserts that there are non-zero numbers $s(v, i)$ such that the map defined by $[T_{v, i}] \mapsto [T_{v, i}]^{s(v, i)}$ induces an embedding of $G(\mathcal{X}_T)$ into $\mathrm{Mod}(S_{g, p})$.  
Then we have an isomorphism $\iota_0 \colon G(\mathcal{X}_T) \cong G(\mathcal{Y})$ by identifying $[T_{v, i}]^{s(v, i)}$ with $\alpha_{v, i}$. 
Note that the power endomorphism 
$$\phi \colon G(\Lambda) \hookrightarrow G(\Lambda)\colon v \mapsto v^{s(v)}$$ 
is an embedding by Lemma \ref{power-homomorphism_injective} (1). 
Here, $s(v):= s(v, 1) \cdots s(v, m_v)$. 
By setting 
$$e(v, i):= e'(v, m_v) s(v) s(v, i)$$ 
and 
$$\psi(v):= [T_{v, 1}]^{e(v, 1)} \cdots [T_{v, m_v}]^{e(v, m_v)},$$ 
we obtain a homomorphism $\psi \colon G(\Lambda) \rightarrow G(\mathcal{Y}) \leq \mathrm{Mod}(S_{g, p})$. 
The homomorphism $\psi$ is injective, because $\phi$, $\psi_0$, $\eta^{-1}$ and $\iota_0$ are injective, and $\psi= \iota_0^{-1} \circ \eta^{-1} \circ \psi_0 \circ \phi$. 
The identity $\psi= \iota_0^{-1} \circ \eta^{-1} \circ \psi_0 \circ \phi$ follows from 
\begin{align*}
 \phi(v) &= v^{s(v)} \\
 &\overset{\psi_0}{\mapsto} ([f_{v, 1}]^{e'(v, 1)} \cdots [f_{v, m_v}]^{e'(v, m_v)})^{s(v)} \\
 &= [f_{v, 1}]^{e'(v, 1)s(v)} \cdots [f_{v, m_v}]^{e'(v, m_v) s(v)} \\
 &\overset{\eta^{-1}}{\mapsto} \alpha_{v, 1}^{e'(v, 1)s(v)} \cdots \alpha_{v, m_v}^{e'(v, m_v) s(v)} \\
 &\overset{\iota_0^{-1}}{\mapsto}  [T_{v, 1}]^{e(v, 1)} \cdots [T_{v, m_v}]^{e(v, m_v)}. 
\end{align*}
\end{proof}

We say that a homomorphism $\psi \colon G(\Lambda) \rightarrow \mathrm{Mod}(S_{g, p})$ satisfies {\it condition (KK)} or {\it Kim--Koberda's condition} if $\psi(v)$ is a product of mutually commutative Dehn twists on $S_{g, p}$ for all $v \in V(\Gamma)$. 
In order to understand homomorphisms with (KK) condition graph theoretically, we introduce the following correspondence between graphs.

\begin{definition}
Let $\Lambda$ and $\Gamma$ be graphs. 
A {\it multi-valued homomorphism} $\phi \colon \Gamma \looparrowright \Lambda$ is a correspondence from $V(\Gamma)$ to $V(\Lambda)$ satisfying the following.  
\begin{enumerate}
 \item[(0)] The vertex-image $\phi(v)$ is a non-empty set of vertices for any $v \in V(\Gamma)$. 
 \item[(1)] If $v_1, v_2 \in V(\Gamma)$ are adjacent, then any pair of vertices $u_1$ and $u_2$, where $u_1 \in \phi(v_1)$ and $u_2 \in \phi(v_2)$, are adjacent. 
\end{enumerate}

A simple example of multi-valued homomorphisms is illustrated in Figure \ref{mv_hom}. 
\begin{figure}
\centering
\includegraphics[clip, scale=0.25]{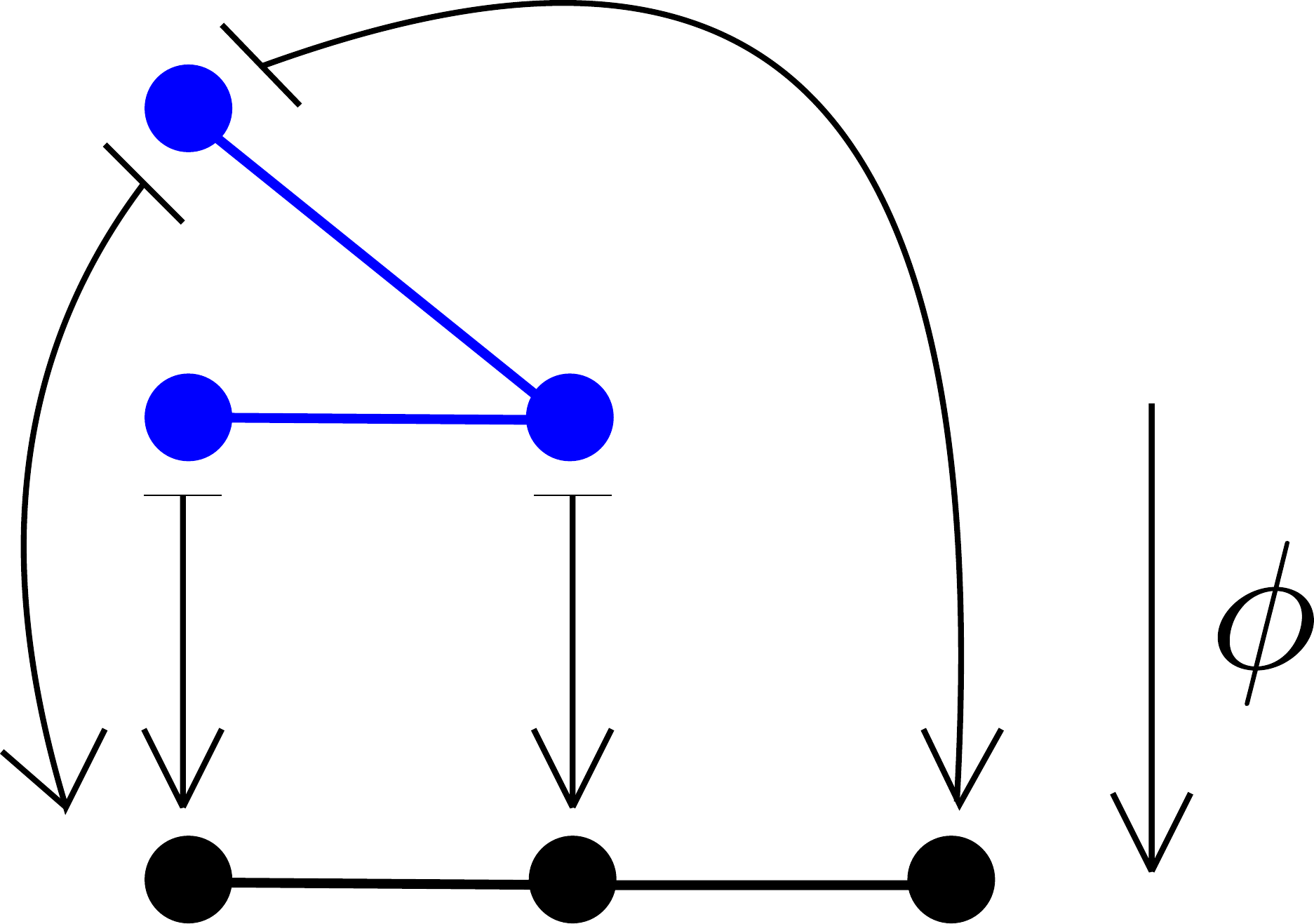}
\caption{A multi-valued homomorphism between two copies of $P_3$, which is not a map. In fact, the uppermost vertex of upper $P_3$ is mapped to the two endpoints of lower $P_3$. }
\label{mv_hom}
\end{figure}

Note that the axiom (1) generalizes the axiom of {\it graph homomorphism}. 
Here, a map $\phi \colon V(\Gamma) \rightarrow V(\Lambda)$ between the vertex sets of two graphs is called a graph homomorphism if $\phi$ maps any pair of adjacent vertices to a pair of adjacent vertices. 
A graph homomorphism satisfies the above axioms (0) and (1), and hence it is a multi-valued homomorphism. 
\end{definition}

Homomorphisms with condition (KK) can be translated into multi-valued homomorphisms.

\begin{proposition}
Suppose that a homomorphism $\psi \colon G(\Lambda) \rightarrow \mathrm{Mod}(S_{g, p})$ satisfies condition (KK). 
Namely, for each vertex $u \in V(\Lambda)$, there are mutually commutative Dehn twists $T_{u, 1}, \ldots, T_{u, m_v}$ and non-zero numbers $e(u, 1), \ldots, e(u, m_u)$ such that 
$$\psi(u) = [T_{u, 1}]^{e(u, 1)} \cdots [T_{u, m_u}]^{e(u, m_u)}. $$ 
Then $\psi$ induces a multi-valued homomorphism $\phi \colon \Gamma \looparrowright \Lambda$ by setting $[T_{u, i}] \mapsto u$. 
Here, $\Gamma$ is the subgraph of $\bar{\mathcal{C}}(S_{g, p})$ induced by the set of closed curves corresponding to the Dehn twists appearing in the presentation of some $\psi(u)$. 
\label{hom_induces_mv}
\end{proposition}
\begin{proof}
We prove that $\phi$ satisfies the axiom (1). 
Pick adjacent vertices $x$ and $y$ in $\Gamma$ and let $[T_{u, i}]$ and $[T_{v, j}]$ be corresponding Dehn twists, where $u$ and $v$ are vertices of $\Lambda$ such that $\phi(x) = u$ and $\phi(y) = v$. 

\begin{claim}
There is an integer $M$ such that the mapping classes $\psi(u)^M$ and $\psi(v)^M$ generate free group of rank $2$ in $\mathrm{Mod}(S_{g, p})$.  
\end{claim}
\begin{proof}
Since $x$ intersects with $y$ non-trivially and minimally as closed curves, the Dehn twists $[T_{u, i}]$ and $[T_{v, j}]$ are non-commutative in $\mathrm{Mod}(S_{g, p})$. 
Hence, by the assumption that $\psi$ satisfies condition (KK), the Dehn twist $[T_{v, j}]$ do not appear in the presentation of $\psi(u)$; 
$$\psi(u) = [T_{u, 1}]^{e(u, 1)} \cdots [T_{u, m_{u}}]^{e(u, m_u)} .$$
As in the proof of Theorem \ref{Kim-Koberda_obst}, we have a right-angled Artin group $G(\mathcal{X}_{T}) \leq \mathrm{Mod}(S_{g, p})$ which is generated by high powers of the Dehn twists appearing in the presentations of the vertex images $\{ \psi(w) \mid w \in V(\Lambda) \}$, where $\mathcal{X}_{T}$ is the anti-commutation graph of those Dehn twists. 
Pick an integer $M$ so that $\psi(u)^M$ and $\psi(v)^M$ are contained in the right-angled Artin group $G(\mathcal{X}_{T})$. 
Then, by Baudisch's theorem \cite[Theorem 1.2]{Baudisch}, $\langle \psi(u)^M, \psi(v)^M \rangle$ is isomorphic to an abelian group or free group of rank $2$. 
Since the Dehn twist $[T_{v, j}]$ do not appear in the presentation of $\psi(u)$, and since $[T_{v, j}]^{Me(v, j)}$ is not commutative with $\psi(u)^M$, it follows that $ \psi(u)^M$ and $\psi(v)^M$ are not commutative in $G(\mathcal{X}_{T})$ (see \cite[Lemma 2.2]{Katayama17-2}). 
Therefore $\langle \psi(u)^M, \psi(v)^M \rangle$ must be a free group of rank $2$. 
\end{proof}

This claim together with the assumption that $\psi$ is a group homomorphism implies that the vertices $u$ and $v$ must be non-commutative in $G(\Lambda)$. 
Namely, $u$ and $v$ are adjacent in $\Lambda$. 
Thus, $\phi$ satisfies the axiom (1). 
\end{proof}

In addition, multi-valued homomorphisms induces homomorphisms with condition (KK).

\begin{proposition}
Let $\Gamma$ be a finite induced subgraph of $\bar{\mathcal{C}}(S_{g, p})$ and $\phi \colon \Gamma \looparrowright  \Lambda$ a multi-valued homomorphism. 
Then, we have a homomorphism $\psi \colon G(\Lambda) \rightarrow \mathrm{Mod}(S_{g, p})$ with condition (KK) by setting 
$$u \mapsto \prod_{x \in \phi^{-1}(u)} [T_x]^{e(u, x)},$$ 
where $T_x$ is the Dehn twist along a closed curve $x$ and $e(u, x)$ is a non-zero integer. 
\label{mv_induces_hom}
\end{proposition}
\begin{proof}
We check that the above $\psi$ is indeed a group homomorphism. 
Pick vertices $u_1$ and $u_2$ with $\{ v_1, v_2 \} \not\in E(\Lambda)$. 
Then $u$ and $v$ are commutative in $G(\Lambda)$. 
By the axiom (1), for any vertex $x$ of $\phi^{-1}(u)$ and any vertex $y$ of $\phi^{-1}(v)$, $x$ and $y$ do not span an edge in $\Gamma$, and so closed curves representing $x$ and $y$ have mutually disjoint representations. 
Hence, the vertices $x$ and $y$ do not span an edge in $\bar{\mathcal{C}}(S_{g, p})$, and $[T_x]$ and $[T_y]$ are commutative. 
It follows that the products $\prod_{x \in \phi^{-1}(u)} [T_x]^{e(u, x)}$ and $\prod_{y \in \phi^{-1}(v)} [T_y]^{e(v, y)}$ are commutative in $\mathrm{Mod}(S_{g, p})$. 
Thus, $\psi$ is a homomorphism.  
In addition, since $\phi$ does not project an edge onto a vertex, $\psi$ satisfies condition (KK). 
\end{proof}

This proposition also shows that, for any multi-valued homomorphism, $\phi \colon \Gamma \looparrowright \Lambda$, between two finite graphs, the map $\psi \colon G(\Lambda) \rightarrow G(\Gamma)$ defined by 
$$u \mapsto \prod_{x \in \phi^{-1}(u)} x, $$
is a homomorphism between corresponding right-angled Artin groups.

\begin{corollary}
Let $p \colon \Gamma \rightarrow \Lambda$ be a  covering map between finite graphs. 
Then there is an embedding $G(\Lambda) \hookrightarrow G(\Gamma)$. 
\end{corollary}
\begin{proof}
We have a homomorphism $F \colon G(\Lambda) \hookrightarrow G(\Lambda) \colon u \mapsto u^{\#p^{-1}(u)}$ which is the composition of the following two maps. 
The first one is a homomorphism 
$$\psi \colon G(\Lambda) \rightarrow G(\Gamma) \colon u \mapsto \prod_{x \in p^{-1}(u)} x$$ 
induced by the covering map $p$ (Proposition \ref{mv_induces_hom}). 
The second one is a homomorphism $\pi_p \colon G(\Gamma) \rightarrow G(\Lambda) \colon x \mapsto p(x)$. 
We will prove that $F = \pi_p \circ \psi$ is an embedding; then so is $\psi$. 
We first check that $\pi_p$ is indeed a (surjective) homomorphism. 
Let $x$ and $y$ be non-adjacent vertices of $\Gamma$. 
Since $p$ is a covering map, $\{ p(x), p(y) \} \in E(\Lambda)$ implies $\{ x, y \} \in E(\Gamma)$. 
Hence, $p(x)$ and $p(y)$ must be non-adjacent. 
Thus, $\pi_p$ is a homomorphism, and so $F$ is also a homomorphism. 
In addition, the surjectivity of the covering map $p$ implies that $\#p^{-1}(u) > 0$ for all $u \in V(\Lambda)$. 
By Lemma \ref{power-homomorphism_injective} (1), it follows that $F$ is an embedding, and therefore $\psi$ is an embedding. 
\end{proof}

\begin{lemma}{(\cite[Theorem 5.4]{Kim-Koberda13})}
Let $\Lambda$ be a finite graph with $P_2 \sqcup P_2 \not\leq \Lambda$ and with the centerless right-angled Artin group $G(\Lambda)$, and let $\Gamma = \Gamma_1 \sqcup \Gamma_2$ be the disjoint union of two finite graphs. 
Suppose that $\psi \colon G(\Lambda) \hookrightarrow G(\Gamma)$ is an embedding. 
Then either at least one homomorphism $\pi_1 \circ \psi \colon G(\Lambda) \rightarrow G(\Gamma_1)$ or $\pi_2 \circ \psi \colon G(\Lambda) \rightarrow G(\Gamma_2)$ is an embedding. 
Here, $\pi_i$ is a natural projection  $G(\Gamma)=G(\Gamma_1) \times G(\Gamma_2) \rightarrow G(\Gamma_i)$ onto the direct factor ($i=1, 2$). 
\label{Kim-Koberda_restrict_direct_factor}
\end{lemma}

An injective graph homomorphism $\iota \colon \Lambda \rightarrow \Gamma$ is said to be a {\it full embedding} if $\iota(\Lambda)$ is an induced subgraph of $\Gamma$.

\begin{proposition}
Let $\phi \colon \Gamma \looparrowright \Lambda$ be the multi-valued homomorphism induced by an embedding $\psi \colon G(\Lambda) \hookrightarrow \mathrm{Mod}(S_{g, p})$ with condition (KK), where $\Gamma$ is a finite induced subgraph of $\bar{\mathcal{C}}(S_{g, p})$. 
Then, for any full embedding $\iota \colon P_n \rightarrow \Lambda$, there is a full embedding $\tilde{\iota} \colon P_n \rightarrow \Gamma$ such that $\iota = \phi \circ \tilde{\iota}$. 
In particular, we have $P_n \leq \bar{\mathcal{C}}(S_{g, p})$. 
\label{path-lifting_lemma}
\end{proposition}
\begin{proof}
Case $n=1$. 
We have only to show that $\psi(u) \neq 1$ for all $u \in V(\Lambda)$, because 
\begin{align*}
\psi(u) \neq 1 & \Leftrightarrow \phi^{-1}(u) \mbox{ is not empty} \\
 & \Leftrightarrow \exists v \in V(\Gamma); \ \phi(v)=u. 
\end{align*} 
However, since $\psi$ is injective, $\psi(u) \neq 1$ for all $u \in V(\Lambda)$. 

Case $n=2$. 
Let $w_1$ and $w_2$ be the endpoints of $P_2$. 
Set $u_i:= \iota(w_i)$ ($i=1, 2$). 
Then the commutator $u_1 u_2 u_1^{-1} u_2^{-1}$ is not the identity element in $G(\Lambda)$. 
Since $\psi$ is injective, $\psi(u_1)$ an $\psi(u_2)$ are not commutative. 
Hence, there are Dehn twists $[T_{i}]$ ($i=1, 2$) such that $[T_{i}]$ appears in the representation of $\psi(u_i)$, and that $[T_1]$ and $[T_2]$ are not commutative. 
We now pick the vertices $v_1$ and $v_2$ of $\Gamma$ corresponding to $[T_1]$ and $[T_2]$, respectively. 
Then $v_1$ and $v_2$ span an edge in $\Gamma$, and that $\phi(v_i) = u_i$ ($i=1, 2$). 
Thus, by setting $w_i \mapsto v_i$, we have a full embedding $\tilde{\iota} \colon P_2 \rightarrow \Gamma$ such that $\iota = \tilde{\iota} \circ \phi$. 

Case $n=3$. 
Using Koberda's embedding theorem, we have $G(\Gamma) \leq \mathrm{Mod}(S_{g, p})$. 
Moreover, the composition of sufficiently large power endomorphism of $G(\Lambda)$ and $\psi$ together with Lemma \ref{power-homomorphism_injective} allows us to have an embedding $\psi \colon G(\Lambda) \hookrightarrow G(\Gamma)$ with condition (KK). 
We label the vertices of $\iota(P_3)$ so that $u_1$ and $u_2$ are adjacent, and that $u_2$ and $u_3$ are adjacent. 
Let $\mathrm{supp}(\psi(u_i))$ ($i=1,2,3$) denotes the subset of $V(\Gamma)$ representing a shortest word of $\psi(u_i)$ with respect to the group presentation $G(\Gamma)$. 
Restricting $\psi$ to the subgroup $G(\iota(P_3))$,  we have an embedding $\psi \colon (G(\iota(P_3))) \rightarrow G(\Gamma)$ satisfying condition (KK). 
Moreover, since $G(\iota(P_3)) \cong \mathbb{Z}*\mathbb{Z}^2$ is centerless and since $P_3$ does not contain $P_2 \sqcup P_2$ as a subgraph, by applying Lemma \ref{Kim-Koberda_restrict_direct_factor} repeatedly, we have a connected component $\Gamma'$ of $\Gamma$ with an embedding $\psi \colon (G(\iota(P_3))) \hookrightarrow G(\Gamma')$ with condition (KK). 
By $\mathrm{supp}(\psi(u_i))$ ($i=1,2,3$), we denote the subset of $V(\Gamma')$ representing a shortest word of $\psi(u_i)$ with respect to the group presentation $G(\Gamma')$ defined in Section \ref{Introduction_section}. 
Note that the subgroup $\langle u_1, u_3 \rangle$ of $G(\iota(P_3))$ is isomorphic to $\mathbb{Z}^2$. 
Therefore $\# (\mathrm{supp}(\psi(u_1)) \cup \mathrm{supp}(\psi(u_3))) \geq 2$. 
Pick a pair of vertices $v_1 \in \mathrm{supp}(\psi(u_1))$ and $v_3 \in \mathrm{supp}(\psi(u_3))$ with $v_1 \neq v_3$. 
Since $\Gamma'$ is connected, there is an edge-path $P$ in $\Gamma'$ connecting $v_1$ with $v_3$. 
The fact that $\psi(u_1)$ and $\psi(u_3)$ are commutative, together with (KK) condition of $\psi$,  implies that any pair of vertices in $ \mathrm{supp}(\psi(u_1)) \cup \mathrm{supp}(\psi(u_3))$ do not span an edge in $\Gamma'$. 
Hence, the edge-path $P$ must contains at least three vertices. 
Then we have a sub-path $(v_1', v_2', v_3') $ such that $v_i' \in \mathrm{supp}(\psi(u_i))$ with $\{v_1', v_2' \},  \{v_2', v_3' \} \in E(\Gamma')$. 
Note that the vertices $v_1'$ and $v_3'$ are commutative in $G(\Gamma')$, because the vertices $u_1$ and $u_3$ are commutative in $G(\iota(P_3))$ (see \cite[Lemma 2.2]{Katayama17-2}). 
Hence, the endpoints $v_1'$ and $v_3'$ of $P$ do not span an edge in $\Gamma'$, and therefore $(v_1', v_2', v_3')$ induces $P_3$ in $\Gamma'$. 
Thus, by setting $\tilde{\iota} \colon P_3 \rightarrow \Gamma' ; u_i \mapsto v_i'$ we have a desired lift of $\iota$.  

Case $n > 3$. 
As in the proof of case $n=3$, by taking a sufficiently high power of $\psi$, we have an embedding $\psi \colon G(\Lambda) \hookrightarrow G(\Gamma)$ with condition (KK). 
By applying \cite[Lemma 4.2]{Katayama17-2}, we have that there is a full embedding $\tilde{\iota} \colon P_n \rightarrow \Gamma$ such that $\iota = \phi \circ \tilde{\iota}$. 
\end{proof}

Cases $n=1$ and $n=2$ in the above lemma say that the multi-valued homomorphism induced by an embedding of a right-angled Artin group is surjective with respect to the vertex set and the edge set. 

\begin{corollary}
Suppose that $\chi(S_{g, p}) < 0$. 
If $G(P_n) \hookrightarrow \mathrm{Mod}(S_{g, p})$, then $P_n \leq \bar{\mathcal{C}}(S_{g, p})$. 
\end{corollary}
\begin{proof}
Suppose that $G(P_n) \hookrightarrow \mathrm{Mod}(S_{g, p})$. 
Then, by Theorem \ref{Kim-Koberda_obst}, we have an embedding $\psi \colon G(P_n) \hookrightarrow \mathrm{Mod}(S_{g, p})$ with condition (KK). 
Proposition \ref{hom_induces_mv} implies that $\psi$ induces a multi-valued homomorphism from a finite induced subgraph $\Gamma$ of $\bar{\mathcal{C}}(S_{g, p})$ to $P_n$. 
Proposition \ref{path-lifting_lemma} now completes this proof. 
\end{proof}

\section{Proofs of theorems \ref{mcg_raag_for_path_graph} and \ref{braid_group_raag} \label{simplest_section}}
We note that even if we use punctures instead of boundary components, the results in Section~\ref{linear_chain_section} also hold. 
Hence we will apply the results in Section~\ref{linear_chain_section} for the orientable surfaces of genus $g$ with $p$ punctures.

We first prove Theorem \ref{mcg_raag_for_path_graph}.

\begin{proof}[Proof of Theorem  \ref{mcg_raag_for_path_graph}] 
Suppose that $S_{g, p}$ is a sphere with $\leq 3$ punctures. 
Then the mapping class group is finite, and therefore embedded right-angled Artin group must be trivial. 

Case $(g, p) \in \{(0, 4), (1, 0), (1,1)\}$. 
Note that $G(P_2) \cong F_2$ (the free group of rank $2$). 
We have an embedding $F_2 \leq \mathbb{Z}/2\mathbb{Z} * \mathbb{Z} / 3 \mathbb{Z} \cong B_3 / \mathbb{Z} \leq \mathrm{Mod}(S_{0, 4})$. 
Moreover, $\mathrm{SL}(2, \mathbb{Z}) = \mathrm{Mod}(S_{1, 0}) \cong \mathrm{Mod}(S_{1, 1})$ contains a subgroup isomorphic to $F_2$. 
Thus in any case we have $F_2 \leq \mathrm{Mod}(S_{g, p})$. 
We next prove that $G(P_m) \leq \mathrm{Mod}(S_{g, p})$ only if $m \leq 2$. 
Note that $\mathrm{Mod}(S_{g, p})$ do not have a subgroup isomorphic to $\mathbb{Z}^2$. 
This shows that $G(P_m) \leq \mathrm{Mod}(S_{g, p})$ implies $m \leq 2$, because $G(P_m)$ contains a subgroup isomorphic to $\mathbb{Z}^2$ when $m \geq 3$. 

The other cases: combine Proposition \ref{path-lifting_lemma} and Theorem \ref{main_thm_linear_chain}. 
\end{proof}

\begin{proposition}\label{mcg_for_raag_cyclic}
$G(C_m) \leq \mathrm{Mod}(S_{g, p})$ if and only if $m$ satisfies: 
\begin{eqnarray*}
m \leq \left\{ \begin{array}{ll}
p & (g=0, p \geq 5) \\
2g+2  & (g \geq 2, p=0)  \\
\end{array} \right.
\end{eqnarray*}
\end{proposition}
\begin{proof}
Koberda's embedding theorem \cite[Theorem 1.1]{Koberda12} together with Figure \ref{cc_gn}
shows that ``if part" of the assertion. 
\begin{figure}
\centering
\includegraphics[clip, scale=0.35]{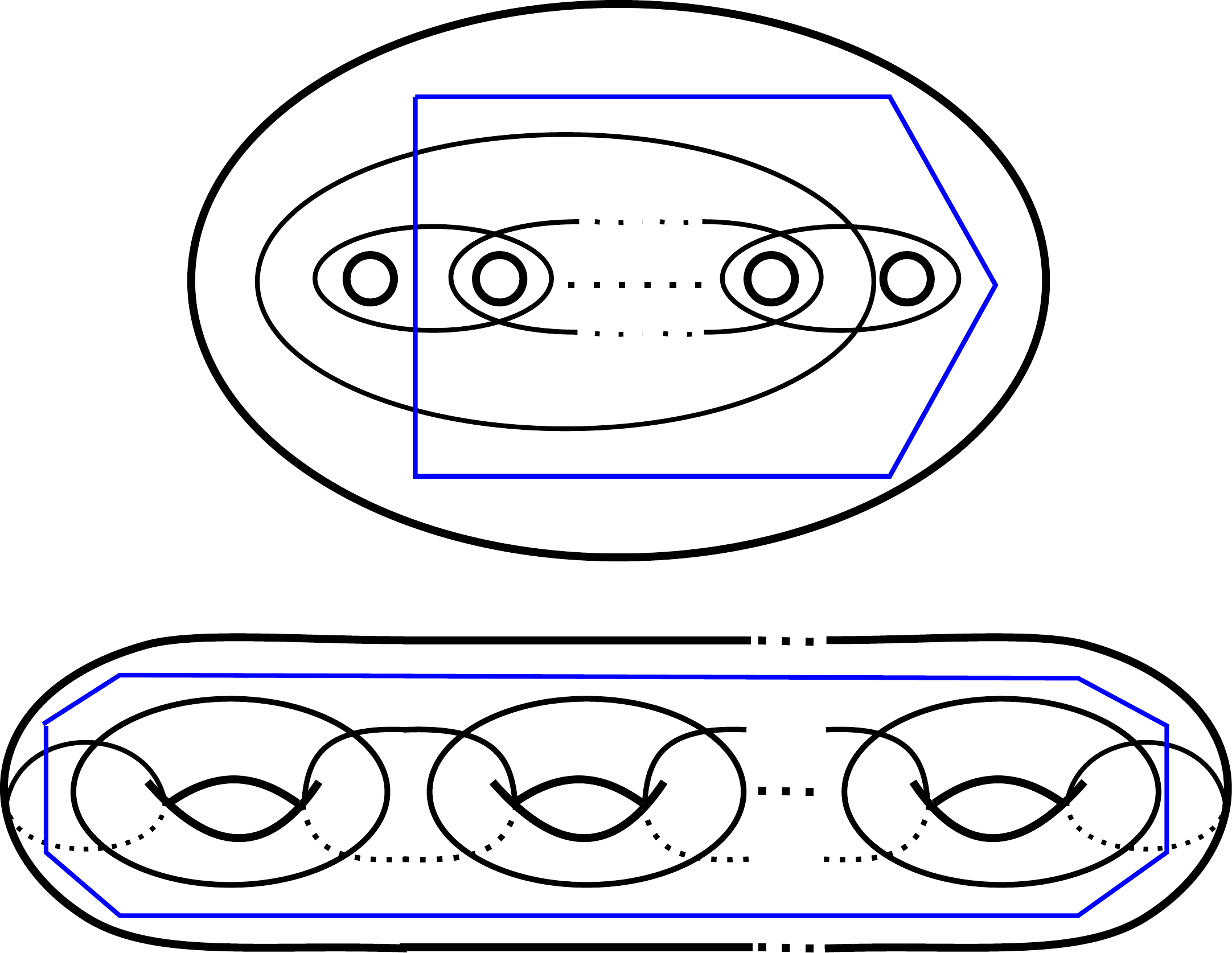}
\caption{$C_{p} \leq \bar{\mathcal{C}}(S_{0}^{p}) = \bar{\mathcal{C}}(S_{0, p})$ and $C_{2g+2} \leq \bar{\mathcal{C}}(S_{g, 0})$. }
\label{cc_gn}
\end{figure}
Let us prove ``only if part". 
Suppose that $G(C_m) \leq \mathrm{Mod}(S_{g, p})$. 
Since $P_{m-1} \leq C_m$, we have $G(P_{m-1}) \leq G(C_m)$. 
Hence, $G(P_{m-1}) \leq \mathrm{Mod}(S_{g, p})$. 
By Theorem \ref{main_thm_linear_chain}, we have either $m-1 \leq p-1$ (if $g=0, \ p\geq 5$) or $m-1 \leq 2g + 1$ (if $g \geq 2, p=0$). 
In either case, we obtain the required result. 
\end{proof}

By $\mathrm{PMod}(S_{g, p})$, we denote the {\it pure mapping class group} of $S_{g, p}$, which is the kernel of the natural homomorphism $\mathrm{Mod}(S_{g, p}) \rightarrow \mathfrak{S}_{p}$. 
Here, $\mathfrak{S}_{p}$ is the permutation group on the set of $p$ punctures $P= \{ \infty_1, \infty_2, \ldots, \infty_p \}$ and the natural homomorphism is induced by the action of $\mathrm{Mod}(S_{g, p})$ to $P$, $\mathrm{Mod}(S_{g, p}) \times P \ni ([f], \infty_i) \mapsto f(\infty_i) \in P$. 
By the definition, the pure mapping class group is a finite index subgroup of the mapping class group. 
We now turn to consider embeddings of $G(P_m)$ and $G(C_m)$ into braid groups.

\begin{lemma}
Let $G$ be a centerless right-angled Artin group. 
Then $G \leq B_p$ if and only if $G \leq \mathrm{Mod}(S_{0, p+1})$. 
\label{braid_mcg_lemma}
\end{lemma}
\begin{proof}
By a result due to Clay--Leininger--Margalit \cite[Theorem 10]{Clay-Leininger-Margalit14}, we have $PB_p \cong \mathrm{PMod}(S_{0, p+1}) \times \mathbb{Z}$. 
Suppose that $G \leq B_p$. 
Then, since $PB_p$ is a finite index subgroup of $B_p$, we have $G \leq PB_p$ by Lemma \ref{power-homomorphism_injective} (2). 
Moreover, since $G$ is centerless, the natural homomorphism $G \rightarrow PB_p / \mathbb{Z} \cong \mathrm{PMod}(S_{0, p+1})$ is an embedding. 
Hence, $G \leq \mathrm{Mod}(S_{0, p+1})$. 

Suppose that $G \leq \mathrm{Mod}(S_{0, p+1})$. 
Then, since $\mathrm{PMod}(S_{0, p+1})$ is a finite index subgroup of $\mathrm{Mod}(S_{0, p+1})$, we have $G \leq \mathrm{PMod}(S_{0, p+1}) \leq PB_p \leq B_p$ by Lemma \ref{power-homomorphism_injective} (2). 
\end{proof}

We now turn to prove Theorem \ref{braid_group_raag}.

\begin{proof}[Proof of Theorem \ref{braid_group_raag}]

(1) Case $p=1$. 
Since $B_1 = 1$, embedded right-angled Artin group must be $G(\emptyset) = 1$. 

Case $p=2$. 
Since $B_2 \cong \mathbb{Z}$, embedded right-angled Artin group must be infinite cyclic. 

Case $p \geq 3$. 
Since the defining graph of any right-angled Artin group with non-trivial center contains an isolated vertex (a vertex non-adjacent to the other vertices) \cite[Lemma 5.1]{Behrstock-Charney12}, the right-angled Artin group $G(P_{m})$ ($m \geq 2$) is centerless. 
By Theorem \ref{mcg_raag_for_path_graph} and Lemma \ref{braid_mcg_lemma}, we obtain the desired result. 

(2) Case $p \leq 2$. 
$B_p$ does not contain a free group of rank $\geq 2$. 
On the other hand, $G(C_m)$ ($m \geq 3$) contains a free group of rank $\geq 2$. 

Case $p=3$. 
$B_3$ is isomorphic to the trefoil knot group. 
\cite[Theorem 1.4 (2)]{Katayama16} gives a classification of right-angled Artin groups embedded in the torus knot groups. 
According to this result, $G(C_3) \cong \langle v_1, v_2, v_3 \mid - \rangle \cong A(V_3)$ is embedded in the trefoil knot group, though $G(C_m)$ ($m \geq 4$) is not embedded in $B_3$. 

Case $p \geq 4$. 
Note that $G(C_m)$ ($m \geq 3$) is centerless. 
An argument similar as in the proof of (1) together with Proposition \ref{mcg_for_raag_cyclic} shows the assertion. 

\end{proof}

\section{Virtual embeddability between mapping class groups \label{virtual_emb_mcg_section}}
In this last section, we introduce some applications of Theorem \ref{mcg_raag_for_path_graph} and \ref{braid_group_raag} to embeddings between finite index subgroups of mapping class groups. 
Recall that $\xi(S_{g, p})$ is the topological complexity of $S_{g, p}$, which is the maximum rank of  the free abelian subgroups in $\mathrm{Mod}(S_{g, p})$ and the quantity $\ell(S_{g, p}) = \ell(S_{g}^{p})$ is the maximum length of the linear chains on $S_{g, p}$, which is computed in Section \ref{linear_chain_section}.

\begin{proof}[Proof of Theorem \ref{for_Birman-Hilden}] 
(1) Suppose that the braid group $B_{2g+1}$ is virtually embedded into $\mathrm{Mod}(S_{g', 0})$. 

Case $g=0$ and $g=1$. 
If $g=0$, then $g'$ always satisfies the desired inequality $g' \geq g$. 
Hence, we assume that $g = 1$. 
Then $B_3$ contains the infinite cyclic subgroup, and hence no finite index subgroup of $B_3$ is embedded in $\mathrm{Mod}(S_{0, 0})$. 
Thus $g' \geq 1$. 

Case $g \geq 2$. 
Let $H$ be a finite index subgroup of $B_{2g+1}$ which is contained in $\mathrm{Mod}(S_{g', 0})$. 
By Theorem \ref{braid_group_raag} (1) and Lemma \ref{power-homomorphism_injective}, $B_{2g+1}$ contains $G(P_{2g+1})$ and therefore $H$ also contains $G(P_{2g+1})$. 
Hence, $G(P_{2g+1}) \leq \mathrm{Mod}(S_{g', 0})$. 
By Theorem \ref{mcg_raag_for_path_graph}, the maximum $m$ such that $G(P_{m}) \leq \mathrm{Mod}(S_{g', 0})$ is equal to one of the following: $0$ (if $g'=0$), $2$ (if $g'= 1$) and $2g'+1$ (if $g' \geq 2$). 
Hence, $g \geq 2$ implies $g' \geq 2$. 
Thus we have $g' \geq g$, as desired.

(2) Suppose that $B_{2g+1}$ is virtually embedded into  $\mathrm{Mod}(S_{g', 0}^{1})$. 

Case $g=0$. 
The inequality $0 \leq g'$ always holds. 

Case $g=1$. 
By Lemma \ref{power-homomorphism_injective}, any finite index subgroup $H$ of $B_3$ contains a free group of rank $2$, and hence $H$ is not embedded into $\mathrm{Mod}(S_{0, 0}^{1}) \cong \mathbb{Z}$. 
Hence, $1 \leq g'$. 

Case $g \geq 2$. 
Suppose that $B_{2g+1}$ is virtually embedded in $\mathrm{Mod}(S_{g', 0}^{1})$. 
By Theorem~\ref{braid_group_raag} (1), we see $G(P_{2g+1}) \leq B_{2g+1}$. 
Hence, $G(P_{2g+1}) \leq \mathrm{Mod}(S_{g', 0}^{1})$ by Lemma \ref{power-homomorphism_injective}. 
Since the kernel of the capping homomorphism 
$$\mathrm{Mod}(S_{g', 0}^{1}) \rightarrow \mathrm{Mod}(S_{g', 1})$$
is the center generated by the Dehn twist along a closed curve which is isotopic into the boundary \cite[Proposition 3.19]{Farb-Margalit12}, and since $G(P_{2g+1})$ is centerless, we have that $G(P_{2g+1}) \leq \mathrm{Mod}(S_{g', 1})$. 
By Theorem \ref{mcg_raag_for_path_graph}, the maximum $m$ such that $G(P_{2g+1}) \leq \mathrm{Mod}(S_{g', 1})$ is one of the following: $0$ (if $g'=0$), $2$ (if $g'= 1$) and $2g'+2$ (if $g' \geq 2$). 
Hence, our assumption $g \geq 2$ implies $g' \geq 2$, and so we have $2g+1 \leq 2g' + 2$. 
Thus we obtain $g \leq g'$, as desired. 

(3) and (4) can be treated similarly. 
\end{proof}

\begin{remark}\label{for_Birman-Hilden_remark}
Note that $B_{2g+1}$ contains $\mathbb{Z}^{2g}$ as a subgroup. 
On the other hand, the maximum rank of the free abelian subgroup of $\mathrm{Mod}(S_{1, 0})$ (resp.\ $\mathrm{Mod}(S_{2,0})$) is $1$ (resp.\ $3$). 
Therefore no finite index subgroup of $B_{3}$ (resp.\ $B_{5}$) is embedded in $\mathrm{Mod}(S_{1, 0})$ (resp.\ $\mathrm{Mod}(S_{2, 0})$). 
Thus the inequality $g' \geq g+1$ holds in the assertion of Theorem \ref{for_Birman-Hilden} (1) when $g=1$ or $2$. 
The authors anticipate that no finite index subgroup of $B_{2g+1}$ is embedded in $\mathrm{Mod}(S_{g', 0})$ for all $g$ and $g'$ with $1 \leq g' \leq g$. 
\end{remark}

Now, Corollary \ref{rigid_mcg} can be deduced from the following theorem.

\begin{theorem}
Suppose that $\chi(S_{g, p})<0$. 
If a finite index subgroup of $\mathrm{Mod}(S_{g, p})$ is embedded in $\mathrm{Mod}(S_{g', p'})$, then the following inequalities hold: 
\begin{enumerate}
 \item[(1)] $\xi(S_{g, p}) \leq \xi(S_{g', p'})$, 
 \item[(2)] $\ell(S_{g, p}) \leq \ell(S_{g', p'})$. 
\end{enumerate}
\label{rigid_mcg_main}
\end{theorem}
\begin{proof}
Suppose that a finite index subgroup $H$ of $\mathrm{Mod}(S_{g, p})$ is embedded in $\mathrm{Mod}(S_{g', p'})$. 
By Theorem \ref{mcg_raag_for_path_graph} (resp.~Birman--Lubotzky--McCarty's result), we can see that $G(P_{\ell(S_{g, p})}) \leq \mathrm{Mod}(S_{g, p})$ (resp.~$\mathbb{Z}^{\xi(S_{g, p})} \leq \mathrm{Mod}(S_{g, p})$). 
Since $H$ is of finite index, we have that $G(P_{\ell(S_{g, p})}) \leq H$ (resp.~$\mathbb{Z}^{\xi(S_{g, p})} \leq H$) by Lemma \ref{power-homomorphism_injective}. 
Hence, $G(P_{\ell(S_{g, p})}) \leq H' \leq \mathrm{Mod}(S_{g', p'})$ (resp.~$\mathbb{Z}^{\xi(S_{g, p})} \leq H' \leq  \mathrm{Mod}(S_{g', p'})$). 
Thus, by Theorem \ref{mcg_raag_for_path_graph} (resp.~Birman--Lubotzky--McCarty's theorem), the desired inequality (2) (resp.~(1)) holds. 
\end{proof}

We finish this paper by proving Theorem \ref{virtual_emb_sphere_closed_surf}.

\begin{proof}[Proof of Theorem \ref{virtual_emb_sphere_closed_surf}] 
Suppose that $p \leq 2g+2$. 
We will prove that $\mathrm{Mod}(S_{0, p})$ is virtually embedded in $\mathrm{Mod}(S_{g, 0})$. 

Case $p \leq 2g+1$. 
We first observe that $B_{2g}$ is embedded in $\mathrm{Mod}(S_{g, 0})$. 
By a theorem due to Birman--Hilden, we can show that $B_{2g}$ is embedded in $\mathrm{Mod}(S_{g-1, 0}^{2})$ as the symmetric subgroup $\mathrm{SMod}(S_{g-1, 0}^{2})$ with respect to a hyper-elliptic involution (see \cite[Chapter 9.4]{Farb-Margalit12}). 
Moreover, the natural surface embedding $j \colon S_{g-1, 0}^{2} \rightarrow S_{g, 0}$, which is obtained by gluing a cylinder to $S_{g-1, 0}^{2}$ along the boundary, induces a homomorphism $j_{*} \colon \mathrm{Mod}(S_{g-1, 0}^{2}) \rightarrow \mathrm{Mod}(S_{g, 0})$.  
A result due to Paris--Rolfsen \cite[Theorem 4.1]{Paris-Rolfsen00} implies that the kernel of $j_{*}$ is generated by $[T_{\beta_1}] [T_{\beta_2}]^{-1}$, where $\beta_1$ and $\beta_2$ are the non-isotopic closed curves parallel to the boundary components of  $S_{g-1, 0}^{2}$ that co-bound a cylinder in $S_{g, 0}$. 
Since $\langle [T_{\beta_1}] [T_{\beta_2}]^{-1} \rangle \cap \mathrm{SMod}(S_{g-1, 0}^{2}) = 1$, we have an embedding $\mathrm{SMod}(S_{g-1, 0}^{2}) \hookrightarrow \mathrm{Mod}(S_{g, 0})$ by restricting $j_{*}$. 
Hence, $B_{2g} \cong \mathrm{SMod}(S_{g-1, 0}^{2}) \hookrightarrow \mathrm{Mod}(S_{g, 0})$. 
Since $B_{2g}$ contains $\mathrm{PMod}(S_{0, p})$, the mapping class group $\mathrm{Mod}(S_{0, p})$ is virtually embedded in $\mathrm{Mod}(S_{g, 0})$.  

Case $p=2g+2$. 
In this case, the symmetric subgroup $\mathrm{SMod}(S_{g, 0})$ of $\mathrm{Mod}(S_{g, 0})$ with respect to a hyper-elliptic involution $\iota$ has a projection onto $\mathrm{Mod}(S_{0, 2g+2})$ with the kernel $\langle \iota \rangle$ (see \cite[Chapter 9.4]{Farb-Margalit12}). 
Consider a finite index subgroup $H$ of $\mathrm{Mod}(S_{g, 0})$ which does not contain $\iota$. 
Residual finiteness of $\mathrm{Mod}(S_{g, 0})$ guarantees the existence of such $H$. 
Then $H \cap \mathrm{SMod}(S_{g, 0})$ is a finite index subgroup of $\mathrm{SMod}(S_{g, 0})$ avoiding $\iota$. 
Hence, $H \cap \mathrm{SMod}(S_{g, 0})$ is embedded in  $\mathrm{Mod}(S_{0, 2g+2})$ as a finite index subgroup. 
Consequently, we have that $\mathrm{Mod}(S_{0, 2g+2})$ is virtually embedded in $\mathrm{Mod}(S_{g, 0})$. 

We next suppose that $\mathrm{Mod}(S_{0, p})$ is virtually embedded in $\mathrm{Mod}(S_{g, 0})$. 
Then, since $\mathrm{Mod}(S_{0, p})$ contains the right-angled Artin group $G(P_{p-1})$ by Theorem \ref{mcg_raag_for_path_graph}, the mapping class group $\mathrm{Mod}(S_{g, 0})$ also contains $G(P_{p-1})$. 
This fact together with Theorem \ref{mcg_raag_for_path_graph} implies that $p-1 \leq 2g+1$, and therefore we have the desired inequality $p \leq 2g+2$. 
\end{proof}


\end{document}